\documentclass[12pt,reqno]{amsart}
\usepackage{amsmath, amssymb, amsthm} 
\usepackage{url}
\usepackage[breaklinks]{hyperref}
\usepackage{tikz}
\usepackage{pgfplots}
\usepackage{rotating, tikz}
\usepackage{subcaption}
\usetikzlibrary{decorations.pathreplacing,calligraphy}
\usetikzlibrary{matrix,arrows,backgrounds}
\usepackage{mathrsfs}
\usepackage{multirow}
\allowdisplaybreaks
\usepackage[headheight=110pt,top=1in, bottom=.9in, left=0.8in, right=0.8in]{geometry}

\numberwithin{equation}{section}
\theoremstyle{plain}
\newtheorem{theorem}{Theorem}[section]

\newtheorem{remark}[]{Remark}
\newtheorem{conjecture}[theorem]{Conjecture}
\newtheorem{corollary}[theorem]{Corollary}

\makeatletter
\def\proof{\@ifnextchar[{\@oproof}{\@nproof}}
\def\@oproof[#1][#2]{\trivlist\item[\hskip\labelsep\textit{#2 Proof of\
		#1.}~]\ignorespaces}
\def\@nproof{\trivlist\item[\hskip\labelsep\textit{Proof.}~]\ignorespaces}

\makeatother
\newcommand{\seqnum}[1]{\href{https://oeis.org/#1}{\rm \underline{#1}}}
\pgfplotsset{compat=1.18}
\begin{document}
\title[Rook decomposition of the Partition function]{Rook decomposition of the Partition function}
\author{N. Guru Sharan }
\address{Department of Mathematics, Indian Institute of Technology Gandhinagar, Palaj, Gandhinagar, Gujarat-382355, India.}
\email{sharanguru5@gmail.com}

\begin{abstract}
The rook numbers are fairly well-studied in the literature. In this paper, we study the max-rook number of the Ferrers boards associated to integer partitions. We show its connections with the Durfee triangle of the partitions. The max-rook number gives a new decomposition of the partition function. We derive the generating functions of the partitions with the Durfee triangle of sizes $3$, $4$ and $5$. We obtain their exact formula and further use it to show the periodicity modulo $p$ for any $p \in \mathbb{N}$ and $p\geq2$. We also establish their parity and parity bias. We give the growth asymptotics of partitions with the Durfee triangle of sizes 3 and 4. We obtain a new rook analogue of the recurrence relation of the partition function.
\end{abstract}
\maketitle
\tableofcontents
\section{Definitions and Notations}

	A \textit{partition} of a non-negative integer $n$ is a non-increasing sequence of positive integers $\lambda_1, \lambda_2, \lambda_3, \ldots, \lambda_r$ such that $\sum_{i=0}^{r} \lambda_i=n$. The individual $\lambda_i$'s are themselves called the parts of the partition. For a fixed $n \in \mathbb{N}$, let $\mathscr{P}(n)$ be the set of all unrestricted partitions of $n$ and let $P(n)$ be the cardinality of $\mathscr{P}(n)$. For example, $(5,3,2,1)$ is a partition of $11$, enumerated by $P(11)$. For convention, we take the empty partition to be the only partition of zero, i.e. $P(0)=1$.

	For each partition $\pi \in \mathscr{P}(n)$, we associate a graphical representation called \emph{Ferrers diagram}, defined by the set of nodes with integer co-ordinates $(i,j)$ such that if $\pi = (\lambda_1, \lambda_2, \lambda_3, ... \ , \lambda_r)$, then the node $(i,j)$ is placed if and only if $-r+1 \leq i \leq 0$ and $0\leq j \leq \lambda_{|i|+1}-1$. For example, the Ferrers diagram of $(5,3,2,1)$ is given in Figure \ref{FigFerDia}.
	
	The largest square of nodes contained in the Ferrers diagram of a partition starting from the top-left node of the Ferrers diagram is called the \emph{Durfee square} and its length is called the \emph{size} of the Durfee square. 
	
	Similarly, let us call the largest right-angled isosceles triangle contained in the Ferrers diagram with the right-angle at the top-left node of the partition, the \emph{Durfee triangle}. The number of nodes on the sides of triangle is called the \emph{size} of the Durfee triangle. One such example of a Durfee square (size $2$) and a Durfee triangle (size $4$), of the partition $(5,3,2,1)$ is given in Figure \ref{FigFerDia}. Durfee triangle of size $n$ contains $\sum_{\ell=1}^{n} \ell = \frac{1}{2}n(n+1)$ many nodes, which is called the $n^{th}$ \emph{triangular number}, which we denote by $t_n$. The first few of them are $1,3,6,10,15, ....$
	
	Formally, a \emph{board} $B$ is any finite sub-lattice of $\mathbb{Z}\times\mathbb{Z}$, containing squares of unit dimension. In the context of this study, we define a \emph{Ferrers board} associated to a partition $\pi$ to be the board that is obtained by replacing each node in the Ferrers diagram of $\pi$ by a unit square. From the defintion of Ferrers Diagram, it is clear that Ferrers Board is a finite connected sub--lattice of $\mathbb{Z}\times\mathbb{Z}$. See Figure \ref{FigFerboa} for clarity.
	
	A  \emph{rook} is a chess piece, which is only allowed to move horizontally or vertically on a board. Two rooks are said to be \emph{non-intersecting} rooks if they are not placed in a same row or a same column, which makes the rooks unable to capture each other, in the game of chess.  In the classical theory of rook numbers, the primary interest is the number of distinct placements of non-intersecting rooks on a board. \emph{The rook number} $R(r,B)$ is the distinct number of ways of placing $r$ many non-intersecting rooks in the Ferrers board $B$. For any fixed board $B$, we define the \textit{rook polynomial} $f_{\textup{B}}(x)$ to be the generating function of the rook numbers $R(r,B)$, i.e.
	\begin{align*}
		f_{\textup{B}}(x):= \sum_{n=1}^{\infty} R(n,B) x^n.
	\end{align*}
	For any fixed board $B$, $R(n,B)$ vanishes for all large $n$, depending on the number of squares in the board, thus making $f_{\textup{B}}(x)$ a polynomial. This term was first coined by Kaplansky and Riordan \cite{KapRio}. The rook polynomial is well-studied. As shown in \cite[Chapter 7]{Riordan}, for a rectangular $m \times n$ board $B$, it is known that
	\begin{align*}
	\hspace{3.5 cm} 	f_{\textup{B}}(x) = n! x^n L_n^{(m-n)}(-x^{-1}), \hspace{2 cm} (m\geq n)
	\end{align*}
	where $L_n^{\alpha}(x)$, the \textit{generalized} \emph{Laguerre} \textit{polynomial}, is
	\begin{align*}
		L_n^{\alpha} (x):= \frac{1}{n!} x^{-\alpha} e^{x} \frac{d^n}{dx^n} \left( x^{n+\alpha} e^{-x} \right),
	\end{align*}
	as defined in \cite[Equation (6.23)]{temme}. The rook polynomial is nothing but a special case of the \emph{matching polynomial} studied in Graph theory, which is the generating function of the number of $k$-edge matchings in a graph. From the theory of matching polynomial, for any fixed board $B$ , it has been shown that the rook numbers $R(n,B)$ are unimodal, i.e. they grow to a maximal value and then decrease. This has been proved as an implication of the fact that all the zeroes of the matching polynomial are real \cite[Section 4]{HL}. The interested reader is referred to \cite[Chapter 2]{BHR} for finer details.

	\section{Introduction}

	We call it a \textit{decomposition of the partition function} when $P(n)$ can be written in terms of a finite sum of other number-theoretic functions. Various decompositions of the partition function have been studied in the literature. For instance, recently, in \cite{merca1} and \cite{merca2}, a decomposition of the partition function in terms of the Euler totient function and the M{\"o}bius function respectively, has been derived.
	
	In this paper, we study the highest possible rook number for a given Ferrers board (which is also the degree of the rook polynomial) and use it to get a new decomposition of the partition function in \eqref{PR_j}, which we call the \textit{rook decomposition}. For a fixed partition $\pi \in \mathscr{P}(n)$, if $B_{\pi}$ be the Ferrers board associated to $\pi$, let us define the \emph{max--rook number} $R(\pi)$ as follows:
	\begin{align*}
		R(\pi):=&\textup{ The maximal number of non-intersecting rooks that can be placed in } B_\pi.
	\end{align*}
	Clearly, the maximal number is unique and less than or equal to $n$, hence $R(\pi)$ is well defined. One example is illustrated in Figure \ref{FigFerboa}. 
\begin{figure}[h!]
		\centering
		\begin{subfigure}[t]{0.4\textwidth}
			\centering
		\begin{tikzpicture}[description/.style={fill=white,inner sep=2pt}]
			\useasboundingbox (-5.5,0.5) rectangle (-1,4);
			\scope[transform canvas={scale=0.8}]
			\filldraw[black] (-6,1.5) circle (3pt);
			\filldraw[black] (-6,2.5) circle (3pt);
			\filldraw[black] (-6,3.5) circle (3pt);
			\filldraw[black] (-6,4.5) circle (3pt);
			\filldraw[black] (-5,2.5) circle (3pt);
			\filldraw[black] (-5,3.5) circle (3pt);
			\filldraw[black] (-5,4.5) circle (3pt);
			\filldraw[black] (-4,3.5) circle (3pt);
			\filldraw[black] (-4,4.5) circle (3pt);
			\filldraw[black] (-3,4.5) circle (3pt);
			\filldraw[black] (-3,4.5) circle (3pt);
			\filldraw[black] (-2,4.5) circle (3pt);
			\draw (-6,4.5) -- (-5,4.5);
			\draw (-6,4.5) -- (-6,3.5);
			\draw (-5,3.5) -- (-5,4.5);
			\draw (-5,3.5) -- (-6,3.5);
			\draw (-6.3,4.8) -- (-2.3,4.8);
			\draw (-6.3,4.8) -- (-6.3,0.8);
			\draw (-6.3,0.8) -- (-2.3,4.8);
			\endscope
		\end{tikzpicture}
		\caption{Ferrers diagram}
		\label{FigFerDia}
		\end{subfigure}
		\begin{subfigure}[t]{0.4\textwidth}
			\centering
		\begin{tikzpicture}[description/.style={fill=white,inner sep=2pt}]
			\useasboundingbox (0.5,0.5) rectangle (5,4);
			\scope[transform canvas={scale=0.8}]
			\node at (1.5,1.5) {\includegraphics[scale=0.1]{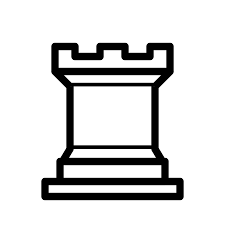}};
			\node at (2.5,2.5) {\includegraphics[scale=0.1]{rook}};
			\node at (3.5,3.5) {\includegraphics[scale=0.1]{rook}};
			\node at (5.5,4.5) {\includegraphics[scale=0.1]{rook}};
			\draw (1,1) -- (2,1);
			\draw (1,2) -- (3,2);
			\draw (1,4) -- (6,4);
			\draw (1,3) -- (4,3);
			\draw (1,5) -- (6,5);
			\draw (1,1) -- (1,5);
			\draw (2,1) -- (2,5);
			\draw (3,2) -- (3,5);
			\draw (4,3) -- (4,5);
			\draw (5,4) -- (5,5);
			\draw (6,4) -- (6,5);
			\endscope
		\end{tikzpicture}
		\caption{Ferrers board}
		\label{FigFerboa}
		\end{subfigure}
		\caption{The partition $(5,3,2,1)$ of $11$ has a max--rook number $4$.}
		\label{ExPic}
	\end{figure}

Then, for $n \in \mathbb{N}$, let us define $F(n)$ as follows:
\begin{align}
	F(n):= \textup{max } \bigg\{ R(\pi) \ \bigg{|} \ \pi \in \mathscr{P}(n) \bigg\}. \label{Fdefn}
\end{align}
For a fixed $n \in \mathbb{N}$, and for any $\pi \in \mathscr{P}(n)$, clearly $R(\pi)\leq n$. Hence, the maximal element in the above set exists, and $F(n)$ is well-defined. For convention, let us define $F(0):=1$.

Next, for any fixed $n, k \in \mathbb{N}$, we define $R_k(n)$, which are the central objects of this study, by
	\begin{align}
		R_k(n) := \# \bigg\{ \pi \in \mathscr{P}(n)\ \bigg|\ R(\pi)=k \bigg\}. \label{RkDef}
	\end{align}
	Thus, $R_k(n)$ is the number of partitions of $n$ whose associated Ferrers board has a max--rook number $k$. Let us denote the generating function of $R_k(n)$ by $\mathscr{F}_k(q)$, i.e.
	\begin{align}
		\mathscr{F}_k(q) := \sum_{n=1}^{\infty} R_k(n) q^n. \label{scrFkDef}
	\end{align} 
	As explained in Theorem \ref{FGF}, the number of partitions with a max--rook number $k$ is equal to number of the partitions with a Durfee triangle of size $k$ and hence $\mathscr{F}_k(q)$ is also the generating function of the number of partitions with a Durfee triangle of size $k$. Since each partition has a unique Durfee triangle, we get the rook decomposition of the partition function as follows: for $n \in \mathbb{N}$,
	\begin{align}
		P(n) = \sum_{k=1}^{n} R_k(n), \label{PR_j}
	\end{align} 
	where $R_k(n)$ is as defined in \eqref{RkDef}. This decomposition is completely combinatorial in nature. Also note that, $R_k(n)=0$ for all $n<t_k$, where $t_n$ is the $n^{th}$ triangular number. From this perspective, the family of functions $R_k(n)$ is the building block of the partition function $P(n)$. Hence, the properties of $R_k(n)$ can give an alternative way to see the properties of the partition function $P(n)$. The decomposition of the partitions up to $P(50)$ has been computed using a PARI algorithm and tabulated in rows and columns by Andrew Howroyd in \seqnum{A325188} \cite{oeis-2}. The decomposition for $P(n)$ for $1\leq n\leq 20$ can be seen in Table \ref{table R} below as the column sum.
	
	\renewcommand{\arraystretch}{1.2}
	\begin{table}[h!]
		\centering
		\begin{tabular}{||c||c|c|c| c| c| c| c| c| c| c| c| c| c| c|c|c|c|c|c|c||} 
			\hline
			$n$ & 1 & 2 & 3 & 4 & 5 & 6 & 7 & 8 & 9 &10 &11 &12& 13 & 14&15&16&17&18&19&20\\
			\hline 
			$P(n)$&  1 & 2 & 3 & 5 & 7 & 11 & 15&22 &30 & 42 &56 & 77 &101& 135 & 176& 231& 297& 385 & 490& 627\\ 
			\hline
			\hline 
			$R_1(n)$& 1 &2 & 2 & 2 & 2 & 2 & 2 &2 & 2 &2 & 2 &2& 2 & 2&2&2&2&2&2&2\\
			\hline
			$R_2(n)$& -- & -- & 1& 3& 5 & 8 & 9 &12 & 13 &16 & 17 &20& 21 & 24&25&28&29&32&33&36\\ 
			\hline
			$R_3(n)$&-- &-- &--  & -- &-- & 1 & 4 & 8 & 15 & 23 & 32 & 43 & 54 & 67 & 82&97&114&133&152&173\\ 
			\hline
			$R_4(n)$&-- &-- &--  & -- &-- & -- & -- &-- & -- & 1 & 5 &12& 24 &42 &66&98&135&181&233&298\\ 
			\hline
			$R_5(n)$&-- &-- &--  & -- &-- & -- & -- &-- & -- &-- & -- &--& -- &--&1&6&17&37&70&118\\ 
			\hline
		\end{tabular}
		\caption{Rook decomposition of the partition function $P(n)$.}
		\label{table R}
	\end{table}
	
	The rows associated with $k=1$ and $k=2$ in Table \ref{table R} represented by $R_1(n)$ and $R_2(n)$ respectively, have been studied. The rows associated to any $k\geq3$ are yet to be studied and the generating functions of the same are not known. In this paper, we prove the generating functions of the rows $k=3$, $4$ and $5$, their exact formulas, periodicity, congruences and asymptotics. 
	
	It can be seen trivially that
	\begin{align}\label{R1E}
		R_1(n) =
		\begin{cases}
			1, & \textup{if } n=1, \\
			2, & \textup{if } n\geq 2,
		\end{cases}
	\end{align} 
	since for $n \geq 2$, only the two partitions $(n)$ and $(1,1, \cdots, 1)$ of $n$ take the value $R(\pi)=1$. 
	Colin Barker  \seqnum{A325168} \cite{oeis} gave the following generating function of $R_2(n)$,
	\begin{align}
		\mathscr{F}_2(q) = \sum_{n=1}^{\infty} R_2(n) q^n = \frac{q^3 (1+2q+q^2+q^3-q^4)}{(1-q)^2(1+q)}. \label{ScrF2GF}
	\end{align}
	$R_2(n)$ satisfies the recurrence relation  \seqnum{A325168} \cite{oeis} for $n>7$ given by,
	\begin{align}
		R_2(n)=R_2(n-1)+R_2(n-2)-R_2(n-3).  \label{R2Rec}
	\end{align}
	In Corollary \ref{R3RecCor}, we prove that $R_3(n)$ satisfies the following analogous recursive relation for $n>14$:
	\begin{align}
		R_3(n) = 2R_3(n-1) -R_3(n-2) +R_3(n-3) -2R_3(n-4) +R_3(n-5). \label{R3RecInt}
	\end{align}
	Further, as observed by Colin Barker \seqnum{A325168} \cite{oeis}, we also know that, for $n>4$,
	\begin{align}\label{R2E}
		R_2(n) =
		\begin{cases}
			2n-4, & \textup{when $n$ is even}, \\
			2n-5, & \textup{when $n$ is odd}.
		\end{cases}
	\end{align}

	The explicit evaluation of $R_2(n)$ in \eqref{R2E} gives the parity for $n>4$: it is even for even arguments and odd for odd arguments. Also, from \eqref{R1E} and \eqref{R2E}, we can see that $R_1(n)$ is bounded and $R_2(n)$ grows linearly respectively, as $n \to \infty$.
	
	In Theorems \ref{R3thmexk} and \ref{R4thmexk}, we find similar explicit fomulas for $R_3(n)$ and $R_4(n)$ respectively, which look surprisingly simple. In principle, one should note that it is possible to use the same method to find explicit formula $R_k(n)$ for any $k \in \mathbb{N}$ but that it becomes cumbersome as $k$ increases. In \cite[Theorem 1.1]{ono}, Bruinier and Ono obtain a closed-form expression for $P(n)$ as a finite sum of some algebraic numbers. Apart from this, there seems to be no known closed-form analytic expression of $P(n)$. The closed-form expressions for $R_k(n)$ can give an analytic closed-form expression for $P(n)$. Also as a direct implication of Theorems \ref{R3thmexk} and \ref{R4thmexk}, we obtain the asymptotic growth of $R_3(n)$ and $R_4(n)$ in Corollaries \ref{CorR3G} and \ref{CorR4G} respectively.

	Ramanujan discovered and proved the following three congruence identities that are satisfied by the partition function, for any $\ell \in \mathbb{N}$:

		\begin{equation} \label{R1C}
			\begin{split}
			&P(5\ell+4)\equiv 0 \textup{ mod (5)},\\
			&P(7\ell+5)\equiv 0 \textup{ mod (7)},\\
			&P(11\ell+6)\equiv 0 \textup{ mod (11)}.	
			\end{split}
		\end{equation}

	The above congruences \eqref{R1C} give us hints concerning the periodic property of $P(n)$ modulo $5$, $7$ and $11$ respectively. We show that $R_k(n)$ for $k=3$ and $4$ exhibit such a periodic property modulo $p$, for any $p\in \mathbb{N}$, $p\geq 2$. Ramanujan's congruences of $P(n)$ could possibly be seen as implications of the periodicity properties of the general $R_k(n)$ mod $(p)$. In \cite{ono2}, Ono studied the distribution of $P(n)$ mod $(m)$ for primes $m\geq5$, which suggests that $R_k(n)$ mod $(p)$ for non-primes $p$ deserves an independent study.
	
	Let $(a;q)_n$ denote the $q$-Pochhamer symbol defined by
		\begin{align*}
				(a;q)_n := (1-a)(1-aq) \cdots (1-aq^{n-1}), \textup{ for } n\geq 1, 
		\end{align*}
	and similarly $(a;q)_{\infty}= (1-a)(1-aq) \cdots$. For convention, let us have $(a;q)_0=1$. 
	
	The generating function of the number of partitions with a Durfee square of size $k$, denoted by $\mathscr{G}_k(q)$, is well studied \cite[p.28]{AndrewsBook} and is given by
	\begin{align} 
		\mathscr{G}_k(q)=\frac{q^{k^2}}{(q;q)_k^2}. \label{ScrGkGF}
	\end{align}
	Unlike $\mathscr{G}_k(q)$, the generating functions $\mathscr{F}_k(q)$ are largely unexplored.	As previously mentioned, we explained in Theorem \ref{FGF} that the number of partitions with a max--rook number $k$ is equal to the number of partitions with a Durfee triangle of size $k$. We exploit this equivalence to get the generating functions of $R_3(n), R_4(n)$ and $R_5(n)$, namely, $\mathscr{F}_3(q)$, $\mathscr{F}_4(q)$ and $\mathscr{F}_5(q)$ in Theorems \ref{F3T}, \ref{F4T} and \ref{F5T} respectively.

	As a direct implication of Euler's pentagonal number theorem, we have the recurrence relation for the partition function $P(n)$ \cite[Corollary 1.8]{AndrewsBook} as follows,  
	\begin{align}
		P(n)= \sum_{k\in \mathbb{Z},\ k\ne 0} (-1)^{k-1} P(n-g_k), \label{PRec}
	\end{align}
	where $g_k=\frac{1}{2}k(3k-1)$ for $k \in \mathbb{Z} \setminus \{0\}$ are called the \emph{generalized pentagonal numbers}. The right-hand side of \eqref{PRec} is effectively a finite sum because $P(n-g_k)=0$ for $k \notin (-\frac{1}{6}(\sqrt{24n+1}-1),\frac{1}{6}(\sqrt{24n+1}-1))$. Thus, \eqref{PRec} is a finite recursive formula for the partition function $P(n)$ but with the number of non-zero terms on the right-hand side depending on $n$, the number being partitioned. In contrast with \eqref{PRec}, the recursive relation for $R_2(n)$ \eqref{R2Rec} has a fixed number of terms (3 terms) on the right-hand side, which is independent of $n$. In Corollary \ref{R3RecCor}, we obtain the recursive relation for $R_3(n)$ and further show that it, too, has a fixed number of terms (5 terms) on the right-hand side. These relations \eqref{R2Rec} and \eqref{R3RecInt} can be seen as rook analogues of \eqref{PRec}.
	
	We explicitly prove the parity and parity bias of $R_3(n)$, $R_4(n)$ and $R_5(n)$ in Section \ref{SecPar}. $R_k(n)$ also satisfies many congruences, apart from modulo $2$. 
	\section{Generating functions}\label{SectionGF}
	The following is the generating function of $F(n)$ defined in \eqref{Fdefn}.
	\begin{theorem}\label{FGF}
		For $|q|<1$,
		\begin{align}
			\sum_{n=0}^{\infty} F(n) q^n = \frac{1}{1-q}\big( (q^2;q^2)_{\infty} (-q;q)_{\infty} -q \big).
		\end{align}
	\end{theorem}
	\begin{proof}
	Fix any arbitrary $n \in \mathbb{N}$, there exists a unique $m \in \mathbb{N}$ such that $t_m \leq n <t_{m+1}$, where $t_m$ is the $m^{th}$ triangular number. Then, there exists a partition of $n$ which has a max--rook number $m$, for instance, $\left(1,2, ...\ , m-1,  n -\frac{m(m-1)}{2}\right)$. Also observe that, for $n<t_{m+1}$, no partition of $n$ can have a max--rook number of $m+1$, since the Durfee triangle of size $m+1$ is necessary for a partition to have a max--rook number $m+1$, which is not possible. Hence, we see,
	\begin{align*}
		F(n) = m, \hspace{0.5 cm} \textup{for any $n$ such that } t_m \leq n <t_{m+1}.
	\end{align*} 
	Thus, the generating function of $F(n)$ is,
	\begin{align}
		\sum_{n=1}^{\infty} F(n) q^n = \sum_{m=1}^{\infty} \sum_{r=0}^{m} m q^{\frac{m(m+1)}{2}+r} = \sum_{m=1}^{\infty} m q^{\frac{m(m+1)}{2}} \sum_{r=0}^{m} q^r = \sum_{m=1}^{\infty} m q^{\frac{m(m+1)}{2}} \times \frac{1-q^{m+1}}{1-q}, \label{FGFstart}
	\end{align}
	where we use the sum of a geometric progression in the last step. Now, we multiply both the sides of \eqref{FGFstart} by $1-q$ to get,
	\begin{align}
		(1-q) \sum_{n=1}^{\infty} F(n) q^n =  \sum_{m=1}^{\infty} m q^{\frac{m(m+1)}{2}} (1-q^{m+1}). \label{Fst1}
	\end{align} 
	We split the terms on both sides into separate series as follows,
	\begin{align*}
		\sum_{n=1}^{\infty} F(n) q^n - \sum_{n=1}^{\infty} F(n) q^{n+1} = \sum_{m=1}^{\infty} m q^{\frac{m(m+1)}{2}} - \sum_{m=1}^{\infty} m q^{\frac{(m+1)(m+2)}{2}}.
	\end{align*}
	Replace $n$ by $n-1$ and $m$ by $m-1$ in the second series of the left-hand side and right-hand side respectively,
	\begin{align*}
	 	\sum_{n=1}^{\infty} F(n) q^n - \sum_{n=2}^{\infty} F(n-1) q^{n} = \sum_{m=1}^{\infty} m q^{\frac{m(m+1)}{2}} - \sum_{m=2}^{\infty} (m-1) q^{\frac{m(m+1)}{2}}.
	\end{align*}
	Now, combine the series to get,
	\begin{align}
		F(1)q + \sum_{n=2}^{\infty} (F(n)-F(n-1)) q^n = q + \sum_{m=2}^{\infty} q^{\frac{m(m+1)}{2}} = q + \sum_{m=2}^{\infty} q^{t_m}, \label{Fst2}
	\end{align}
	where $t_m$ is the $m^{th}$ triangular number. Hence, $F(1)=1$, and for $n\geq2$,
	\begin{align}\label{FExp}
		F(n)-F(n-1)= 
		\begin{cases}
			1,& n=t_m \textup{ for some $m \in \mathbb{N}$},\\
			0,& \textup{otherwise.} 
		\end{cases}
	\end{align}
	From \eqref{Fst1} and \eqref{Fst2}, we have,
	\begin{align}
		\sum_{n=1}^{\infty} F(n) q^n = \frac{1}{1-q} \sum_{m=1}^{\infty} q^{\frac{m(m+1)}{2}}. \label{Fst3}
	\end{align}
	Use \cite[Corollary 2.10]{AndrewsBook} and add $F(0)=1$ to both sides and simplify to get the required result.
	\end{proof}
	Now, we give a combinatorial proof of the generating function of $R_3(n)$ as a Theorem.
	\begin{theorem}\label{F3T}
		For $|q|<1$, we have,
		\begin{align}
			\mathscr{F}_3(q)= \sum_{n=1}^{\infty} R_3(n) q^n = \frac{q^6(1+2q+q^2+2q^3-q^4-q^6-q^7+q^8)}{(1-q)^3(1+q+q^2)}. \label{ScrF3GF}
 		\end{align}
	\end{theorem}
	\begin{proof}
		The partitions of a number with the Durfee triangle of size $3$ are either of the type A, B, their conjugates A$^c$ or B$^c$, shown in Figure \ref{R_3Pic}. We have the conditions $m,s\geq0$ and $\ell\leq m+1$ for type A, and $s\leq\ell+1$ and $\ell\leq m+1$, along with $s,\ell\geq1$ for type B. We take $s,\ell\geq 1$ since any partition of type B or B$^c$ with $\ell=0$ or $s=0$ is either a partition of type A or A$^c$. We find the generating functions of partitions of type A, B, A$^c$ and B$^c$ separately and add them. We then subtract off the partitions that are repeated to get the required generating function $\mathscr{F}_3(q)$. 
		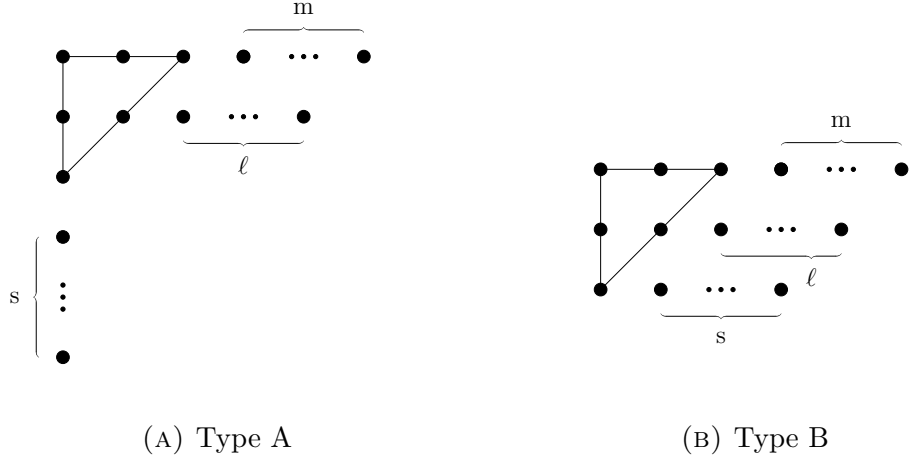
\begin{figure*}[h!]
			\centering
			\begin{subfigure}[t]{0.4\textwidth}
				\centering
					\begin{tikzpicture}[description/.style={fill=white,inner sep=2pt}]
					\useasboundingbox (-5.5,-1) rectangle (0,4.5);
					\scope[transform canvas={scale=0.8}]
					\filldraw[black] (-6,1.5) circle (3pt);
					\filldraw[black] (-6,0.3) circle (1pt);
					\filldraw[black] (-6,0.5) circle (1pt);
					\filldraw[black] (-6,0.7) circle (1pt);
					\filldraw[black] (-6,-0.5) circle (3pt);
					\filldraw[black] (-6,2.5) circle (3pt);
					\filldraw[black] (-6,3.5) circle (3pt);
					\filldraw[black] (-6,4.5) circle (3pt);
					\filldraw[black] (-3.2,3.5) circle (1pt);
					\filldraw[black] (-3,3.5) circle (1pt);
					\filldraw[black] (-2.8,3.5) circle (1pt);
					\filldraw[black] (-2,3.5) circle (3pt);
					\filldraw[black] (-5,3.5) circle (3pt);
					\filldraw[black] (-5,4.5) circle (3pt);
					\filldraw[black] (-4,3.5) circle (3pt);
					\filldraw[black] (-4,4.5) circle (3pt);
					\filldraw[black] (-3,4.5) circle (3pt);
					\filldraw[black] (-3,4.5) circle (3pt);
					\filldraw[black] (-2.2,4.5) circle (1pt);
					\filldraw[black] (-2,4.5) circle (1pt);
					\filldraw[black] (-1.8,4.5) circle (1pt);
					\filldraw[black] (-1,4.5) circle (3pt);
					\draw (-6,4.5) -- (-4,4.5);
					\draw (-6,4.5) -- (-6,2.5);
					\draw (-4,4.5) -- (-6,2.5);
					\node at (-2,5.3) {m};
					\node at (-3,2.7) {$\ell$};
					\node at (-6.8,0.5) {s};
					\draw [decorate, 
						decoration = {calligraphic brace,
						raise=0pt,
						aspect=0.5}] (-3,4.9) --  (-1,4.9);
					\draw [decorate, 
						decoration = {calligraphic brace, mirror, 
						raise=0pt,
						aspect=0.5}] (-4,3.1) --  (-2,3.1);
					\draw [decorate, 
						decoration = {calligraphic brace,
						raise=0pt,
						aspect=0.5}] (-6.4,-0.5) --  (-6.4,1.5);
					\endscope
 				\end{tikzpicture}
				\caption{Type A}
			\end{subfigure}
			~ 
			\begin{subfigure}[t]{0.4\textwidth}
				\centering
				\begin{tikzpicture}[description/.style={fill=white,inner sep=2pt}]
					\useasboundingbox (-5.5,0.5) rectangle (0,4.5);
					\scope[transform canvas={scale=0.8}]
					\filldraw[black] (-6,2.5) circle (3pt);
					\filldraw[black] (-6,3.5) circle (3pt);
					\filldraw[black] (-6,4.5) circle (3pt);
					\filldraw[black] (-5,2.5) circle (3pt);
					\filldraw[black] (-5,3.5) circle (3pt);
					\filldraw[black] (-5,4.5) circle (3pt);
					\filldraw[black] (-4.2,2.5) circle (1pt);
					\filldraw[black] (-4,2.5) circle (1pt);
					\filldraw[black] (-3.8,2.5) circle (1pt);
					\filldraw[black] (-3,2.5) circle (3pt);
					\filldraw[black] (-4,3.5) circle (3pt);
					\filldraw[black] (-4,4.5) circle (3pt);
					\filldraw[black] (-3.2,3.5) circle (1pt);
					\filldraw[black] (-3,3.5) circle (1pt);
					\filldraw[black] (-2.8,3.5) circle (1pt);
					\filldraw[black] (-2,3.5) circle (3pt);
					\filldraw[black] (-3,4.5) circle (3pt);
					\filldraw[black] (-3,4.5) circle (3pt);
					\filldraw[black] (-2.2,4.5) circle (1pt);
					\filldraw[black] (-2,4.5) circle (1pt);
					\filldraw[black] (-1.8,4.5) circle (1pt);
					\filldraw[black] (-1,4.5) circle (3pt);
					\draw (-6,4.5) -- (-4,4.5);
					\draw (-6,4.5) -- (-6,2.5);
					\draw (-4,4.5) -- (-6,2.5);
					\node at (-2,5.3) {m};
					\node at (-2.5,2.7) {$\ell$};
					\node at (-4,1.7) {s};
					\draw [decorate, 
					decoration = {calligraphic brace,
						raise=0pt,
						aspect=0.5}] (-3,4.9) --  (-1,4.9);
					\draw [decorate, 
					decoration = {calligraphic brace, mirror, 
						raise=0pt,
						aspect=0.75}] (-4,3.1) --  (-2,3.1);
					\draw [decorate, 
					decoration = {calligraphic brace, mirror,
						raise=0pt,
						aspect=0.5}] (-5,2.1) --  (-3,2.1);
					\endscope
				\end{tikzpicture}
				\caption{Type B}
			\end{subfigure}
			\caption{The two types of partitions with Durfee triangle of size $3$.}
			\label{R_3Pic}
		\end{figure*}
		\begin{align}
			&\hspace{-0.5 cm}\textup{The generating function of partitions of type A} \notag \\
			&= q^6 \sum_{m=0}^{\infty} q^m \sum_{\ell=0}^{m+1} q^\ell \sum_{s=0}^{\infty} q^s \notag \\
			&= \frac{q^6}{1-q} \sum_{m=0}^{\infty} \frac{q^m(1-q^{m+2})}{1-q} \notag \\
			&=  \frac{q^6}{(1-q)^2} \sum_{m=0}^{\infty} q^m -  \frac{q^8}{(1-q)^2} \sum_{m=0}^{\infty} q^{2m} \notag \\
			&= \frac{q^6}{(1-q)^3} - \frac{q^8}{(1-q)^2(1-q^2)}. \label{R3TA}\\ \notag
			&\hspace{-0.5 cm}\textup{The generating function of partitions of type B} \notag \\
			&= q^6 \sum_{m=0}^{\infty} q^m \sum_{\ell=1}^{m+1} q^\ell \sum_{s=1}^{\ell+1} q^s  \notag \\
			&= q^7 \sum_{m=0}^{\infty} q^m \sum_{\ell=1}^{m+1} \frac{q^\ell(1-q^{\ell+1})}{1-q}  \notag \\
			&= \frac{q^7}{1-q} \sum_{m=0}^{\infty} q^m \sum_{\ell=1}^{m+1}q^{\ell} -  \frac{q^8}{1-q} \sum_{m=0}^{\infty} q^m \sum_{\ell=1}^{m+1}q^{2\ell}  \notag \\
			&= \frac{q^8}{1-q} \sum_{m=0}^{\infty} \frac{q^m(1-q^{m+1})}{1-q} - \frac{q^{10}}{1-q} \sum_{m=0}^{\infty} \frac{q^m(1-q^{2m+2})}{1-q^2}  \notag \\
			&=  \frac{q^8}{(1-q)^2} \sum_{m=0}^{\infty} q^m - \frac{q^9}{(1-q)^2} \sum_{m=0}^{\infty} q^{2m}   - \frac{q^{10}}{(1-q)(1-q^2)} \sum_{m=0}^{\infty} q^m + \frac{q^{12}}{(1-q)(1-q^2)} \sum_{m=0}^{\infty} q^{3m}  \notag \\
			&= \frac{q^8}{(1-q)^3} - \frac{q^9}{(1-q)^2(1-q^2)} - \frac{q^{10}}{(1-q)^2(1-q^2)} +\frac{q^{12}}{(1-q)(1-q^2)(1-q^3)}.\label{R3TB}
		\end{align}

	The generating functions of type A$^c$ and B$^c$ are the same as A and B respectively, hence we multiply \eqref{R3TA} and \eqref{R3TB} by 2. The partitions which are both of type A and A$^c$, i.e. the ones that are counted twice, are exactly when $\ell=0$, which we subtract. The only partitions which are of type B and B$^c$ are $(3,3,2)$ and $(3,3,3)$, hence we subtract $q^8$ and $q^9$ to get the generating function $\mathscr{F}_3(q)$. Thus,
	\begin{align*}
		\mathscr{F}_3(q) &= 2 \left( \frac{q^6}{(1-q)^3} - \frac{q^8}{(1-q)^2(1-q^2)} \right) + 2 \left( \frac{q^8}{(1-q)^3} - \frac{q^9}{(1-q)^2(1-q^2)} - \frac{q^{10}}{(1-q)^2(1-q^2)} \right. \\
		& \quad \left. +\frac{q^{12}}{(1-q)(1-q^2)(1-q^3)} \right) - q^6 \sum_{m=1}^{\infty} q^{m}\sum_{s=1}^{\infty} q^{s} -q^8 -q^9.
	\end{align*} 
	Simplify to get \eqref{ScrF3GF}.
	\end{proof}
	We use the generating function of $R_3(n)$ to obtain its recurrence relation analogous to that of $R_2(n)$ as stated in \eqref{R2Rec}.
	\begin{corollary}\label{R3RecCor}
		For $n>14$, the following recurrence relation holds:
		\begin{align}
			R_3(n) = 2R_3(n-1) -R_3(n-2) +R_3(n-3) -2R_3(n-4) +R_3(n-5).  \label{R3Rec}
		\end{align}
	\end{corollary} 
	\begin{proof}
		From Theorem \ref{F3T}, we have
		\begin{align}
			\sum_{n=1}^{\infty} R_3(n) q^n 	&= \frac{q^6(1+2q+q^2+2q^3-q^4-q^6-q^7+q^8)}{1-2q+q^2-q^3+2q^4-q^5}.
		\end{align}
	We cross-multiply to get,
		\begin{align*}
			\left( \sum_{n=1}^{\infty} R_3(n) q^n \right) \left( 1-2q+q^2-q^3+2q^4-q^5 \right) = q^6(1+2q+q^2+2q^3-q^4-q^6-q^7+q^8).
		\end{align*}	
	Hence, we get,
	\begin{align*}
		&\sum_{n=1}^{\infty} R_3(n) q^n - \sum_{n=1}^{\infty} 2 R_3(n) q^{n+1} + \sum_{n=1}^{\infty} R_3(n) q^{n+2}- \sum_{n=1}^{\infty} R_3(n) q^{n+3}+ \sum_{n=1}^{\infty} 2 R_3(n) q^{n+4} - \sum_{n=1}^{\infty} R_3(n) q^{n+5}\\
		&= q^6(1+2q+q^2+2q^3-q^4-q^6-q^7+q^8).
 	\end{align*}
 	Use the fact that $R_3(n)=0$ for $1\leq n\leq5$ to get,
 	\begin{align*}
 		&\sum_{n=6}^{\infty} R_3(n) q^n - \sum_{n=5}^{\infty} 2 R_3(n) q^{n+1} + \sum_{n=4}^{\infty} R_3(n) q^{n+2}- \sum_{n=3}^{\infty} R_3(n) q^{n+3}+ \sum_{n=2}^{\infty} 2 R_3(n) q^{n+4} - \sum_{n=1}^{\infty} R_3(n) q^{n+5}\\
 		&= q^6(1+2q+q^2+2q^3-q^4-q^6-q^7+q^8).
 	\end{align*}
 	Taking a suitable change of variables in each of the series on the left-hand side, we get,
 	\begin{align*}
 		&\sum_{n=6}^{\infty}\bigg( R_3(n) - 2R_3(n-1) +R_3(n-2) -R_3(n-3) +2R_3(n-4) -R_3(n-5) \bigg) q^n\\
 		&= q^6(1+2q+q^2+2q^3-q^4-q^6-q^7+q^8).
 	\end{align*}
 	Hence, by comparing the coefficients of $q^n$ on both sides of the equation, we get, for $n>14$,
 	\begin{align*}
 		R_3(n) - 2R_3(n-1) +R_3(n-2) -R_3(n-3) +2R_3(n-4) -R_3(n-5) =0.
 	\end{align*}
 	Rearrange the terms to get the required recurrence relation \eqref{R3Rec}.
	\end{proof} 
	Now, we give a combinatorial proof of the generating function of $R_4(n)$ as a Theorem.
	\begin{theorem}\label{F4T}
	For $|q|<1$, we have,
		\begin{align}
			\mathscr{F}_4(q) &= \sum_{n=1}^{\infty} R_4(n) q^n \notag \\
			&= \frac{q^{10}(1+4q+6q^2+7q^3+6q^4+2q^5-5q^7-5q^8-5q^9+q^{11}+3q^{12}+2q^{13}-q^{16})}{(1-q)(1-q^2)(1-q^3)(1-q^4)}. \label{ScrF4GF}
		\end{align}		
	\end{theorem}
	\begin{proof}
		The partitions of a number with the Durfee triangle of size $4$ are at least of one of the type A, B, C, their conjugates A$^c$, B$^c$ or C$^c$, shown in Figure \ref{R_4Pic}. For type A, we take $s,m\geq0$, $r\leq \ell+1$ and $\ell \leq m+1$. For type B, we take $s\geq0$, $m\geq0$, $r\geq1$, $\ell\geq1$, $r\leq s+1$ and $\ell\leq m+1$, since $r=0$ or $\ell=0$ is already covered by type A and A$^c$. For type C, we take $m \geq 0$, $s\geq1$, $r\geq 1$, $\ell\geq1$, $s\leq r+1$, $r\leq \ell+1$ and $\ell\leq m+1$, since $s=0$, $r=0$ and $\ell=0$ is covered by type A, type B and type A$^c$ respectively. We now find the generating functions of the types separately.  
		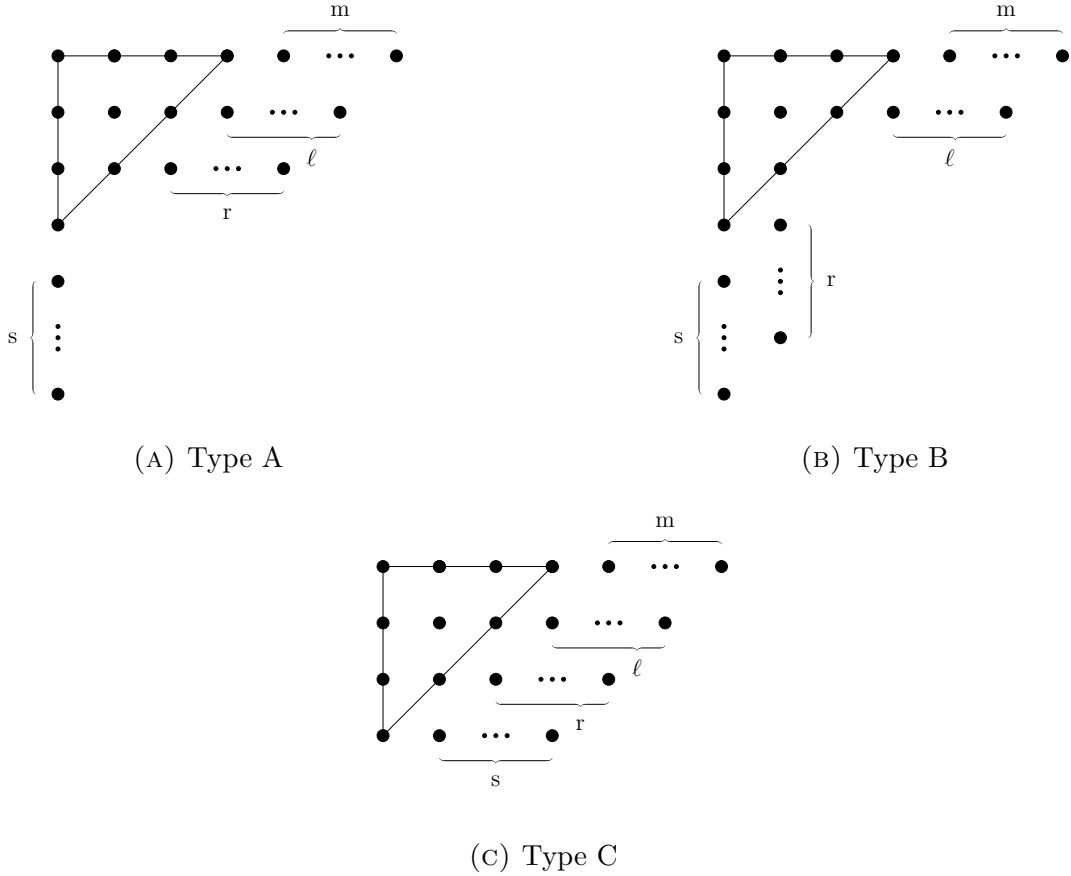
\begin{figure*}[h!]
			\centering
			\begin{subfigure}[t]{0.5\textwidth}
				\centering
				\begin{tikzpicture}[description/.style={fill=white,inner sep=2pt}]
					\useasboundingbox (-5.5,-1.5) rectangle (0.5,4);
					\scope[transform canvas={scale=0.75}]
					\filldraw[black] (-6,1.5) circle (3pt);
					\filldraw[black] (-6,0.5) circle (3pt);
					\filldraw[black] (-6,-0.3) circle (1pt);
					\filldraw[black] (-6,-0.5) circle (1pt);
					\filldraw[black] (-6,-0.7) circle (1pt);
					\filldraw[black] (-6,-1.5) circle (3pt);
					\filldraw[black] (-6,2.5) circle (3pt);
					\filldraw[black] (-6,3.5) circle (3pt);
					\filldraw[black] (-6,4.5) circle (3pt);
					\filldraw[black] (-5,4.5) circle (3pt);
					\filldraw[black] (-5,2.5) circle (3pt);
					\filldraw[black] (-3.2,2.5) circle (1pt);
					\filldraw[black] (-3,2.5) circle (1pt);
					\filldraw[black] (-2.8,2.5) circle (1pt);
					\filldraw[black] (-4,2.5) circle (3pt);
					\filldraw[black] (-2,2.5) circle (3pt);
					\filldraw[black] (-2.2,3.5) circle (1pt);
					\filldraw[black] (-2,3.5) circle (1pt);
					\filldraw[black] (-1.8,3.5) circle (1pt);
					\filldraw[black] (-1,3.5) circle (3pt);
					\filldraw[black] (-3,3.5) circle (3pt);
					\filldraw[black] (-5,3.5) circle (3pt);
					\filldraw[black] (-5,4.5) circle (3pt);
					\filldraw[black] (-4,3.5) circle (3pt);
					\filldraw[black] (-4,4.5) circle (3pt);
					\filldraw[black] (-3,4.5) circle (3pt);
					\filldraw[black] (-3,4.5) circle (3pt);
					\filldraw[black] (-1.2,4.5) circle (1pt);
					\filldraw[black] (-1,4.5) circle (1pt);
					\filldraw[black] (-0.8,4.5) circle (1pt);
					\filldraw[black] (0,4.5) circle (3pt);
					\filldraw[black] (-2,4.5) circle (3pt);
					\draw (-6,4.5) -- (-3,4.5);
					\draw (-6,4.5) -- (-6,1.5);
					\draw (-3,4.5) -- (-6,1.5);
					\node at (-1,5.3) {m};
					\node at (-1.5,2.7) {$\ell$};
					\node at (-3,1.7) {r};
					\node at (-6.8,-0.5) {s};
					\draw [decorate, 
					decoration = {calligraphic brace,
						raise=0pt,
						aspect=0.5}] (-2,4.9) --  (0,4.9);
					\draw [decorate, 
					decoration = {calligraphic brace, mirror, 
						raise=0pt,
						aspect=0.75}] (-3,3.1) --  (-1,3.1);
					\draw [decorate, 
					decoration = {calligraphic brace, mirror, 
						raise=0pt,
						aspect=0.5}] (-4,2.1) --  (-2,2.1);
					\draw [decorate, 
					decoration = {calligraphic brace,
						raise=0pt,
						aspect=0.5}] (-6.4,-1.5) --  (-6.4,0.5);
					\endscope
				\end{tikzpicture}
				\caption{Type A}
			\end{subfigure}
			~ 
			\begin{subfigure}[t]{0.5\textwidth}
				\centering
				\begin{tikzpicture}[description/.style={fill=white,inner sep=2pt}]
					\useasboundingbox (-5.5,-1.5) rectangle (0.5,4);
					\scope[transform canvas={scale=0.75}]
					\filldraw[black] (-6,1.5) circle (3pt);
					\filldraw[black] (-6,0.5) circle (3pt);
					\filldraw[black] (-6,-0.3) circle (1pt);
					\filldraw[black] (-6,-0.5) circle (1pt);
					\filldraw[black] (-6,-0.7) circle (1pt);
					\filldraw[black] (-6,-1.5) circle (3pt);
					\filldraw[black] (-6,2.5) circle (3pt);
					\filldraw[black] (-6,3.5) circle (3pt);
					\filldraw[black] (-6,4.5) circle (3pt);
					\filldraw[black] (-5,4.5) circle (3pt);
					\filldraw[black] (-5,2.5) circle (3pt);
					\filldraw[black] (-5,0.7) circle (1pt);
					\filldraw[black] (-5,0.5) circle (1pt);
					\filldraw[black] (-5,0.3) circle (1pt);
					\filldraw[black] (-5,1.5) circle (3pt);
					\filldraw[black] (-5,-0.5) circle (3pt);
					\filldraw[black] (-2.2,3.5) circle (1pt);
					\filldraw[black] (-2,3.5) circle (1pt);
					\filldraw[black] (-1.8,3.5) circle (1pt);
					\filldraw[black] (-1,3.5) circle (3pt);
					\filldraw[black] (-3,3.5) circle (3pt);
					\filldraw[black] (-5,3.5) circle (3pt);
					\filldraw[black] (-5,4.5) circle (3pt);
					\filldraw[black] (-4,3.5) circle (3pt);
					\filldraw[black] (-4,4.5) circle (3pt);
					\filldraw[black] (-3,4.5) circle (3pt);
					\filldraw[black] (-3,4.5) circle (3pt);
					\filldraw[black] (-1.2,4.5) circle (1pt);
					\filldraw[black] (-1,4.5) circle (1pt);
					\filldraw[black] (-0.8,4.5) circle (1pt);
					\filldraw[black] (0,4.5) circle (3pt);
					\filldraw[black] (-2,4.5) circle (3pt);
					\draw (-6,4.5) -- (-3,4.5);
					\draw (-6,4.5) -- (-6,1.5);
					\draw (-3,4.5) -- (-6,1.5);
					\node at (-1,5.3) {m};
					\node at (-2,2.7) {$\ell$};
					\node at (-4.1,0.5) {r};
					\node at (-6.8,-0.5) {s};
					\draw [decorate, 
					decoration = {calligraphic brace,
						raise=0pt,
						aspect=0.5}] (-2,4.9) --  (0,4.9);
					\draw [decorate, 
					decoration = {calligraphic brace, mirror, 
						raise=0pt,
						aspect=0.55}] (-3,3.1) --  (-1,3.1);
					\draw [decorate, 
					decoration = {calligraphic brace, mirror, 
						raise=0pt,
						aspect=0.5}] (-4.5,-0.5) --  (-4.5,1.5);
					\draw [decorate, 
					decoration = {calligraphic brace,
						raise=0pt,
						aspect=0.5}] (-6.4,-1.5) --  (-6.4,0.5);
					\endscope
				\end{tikzpicture}
				\caption{Type B}
			\end{subfigure}
			\\
			~ 
			\begin{subfigure}[t]{0.3\textwidth}
								\centering						\begin{tikzpicture}[description/.style={fill=white,inner sep=2pt}]
									\useasboundingbox (-5,0) rectangle (0.5,4.5);
									\scope[transform canvas={scale=0.75}]
									\filldraw[black] (-6,1.5) circle (3pt);
									\filldraw[black] (-6,2.5) circle (3pt);
									\filldraw[black] (-6,3.5) circle (3pt);
									\filldraw[black] (-6,4.5) circle (3pt);
									\filldraw[black] (-5,4.5) circle (3pt);
									\filldraw[black] (-5,2.5) circle (3pt);
									\filldraw[black] (-3.2,2.5) circle (1pt);
									\filldraw[black] (-3,2.5) circle (1pt);
									\filldraw[black] (-2.8,2.5) circle (1pt);
									\filldraw[black] (-4,2.5) circle (3pt);
									\filldraw[black] (-2,2.5) circle (3pt);
									\filldraw[black] (-4.2,1.5) circle (1pt);
									\filldraw[black] (-4,1.5) circle (1pt);
									\filldraw[black] (-3.8,1.5) circle (1pt);
									\filldraw[black] (-5,1.5) circle (3pt);
									\filldraw[black] (-3,1.5) circle (3pt);
									\filldraw[black] (-2.2,3.5) circle (1pt);
									\filldraw[black] (-2,3.5) circle (1pt);
									\filldraw[black] (-1.8,3.5) circle (1pt);
									\filldraw[black] (-1,3.5) circle (3pt);
									\filldraw[black] (-3,3.5) circle (3pt);
									\filldraw[black] (-5,3.5) circle (3pt);
									\filldraw[black] (-5,4.5) circle (3pt);
									\filldraw[black] (-4,3.5) circle (3pt);
									\filldraw[black] (-4,4.5) circle (3pt);
									\filldraw[black] (-3,4.5) circle (3pt);
									\filldraw[black] (-3,4.5) circle (3pt);
									\filldraw[black] (-1.2,4.5) circle (1pt);
									\filldraw[black] (-1,4.5) circle (1pt);
									\filldraw[black] (-0.8,4.5) circle (1pt);
									\filldraw[black] (0,4.5) circle (3pt);
									\filldraw[black] (-2,4.5) circle (3pt);
									\draw (-6,4.5) -- (-3,4.5);
									\draw (-6,4.5) -- (-6,1.5);
									\draw (-3,4.5) -- (-6,1.5);
									\node at (-1,5.3) {m};
									\node at (-1.5,2.7) {$\ell$};
									\node at (-2.5,1.7) {r};
									\node at (-4,0.7) {s};
									\draw [decorate, 
									decoration = {calligraphic brace,
										raise=0pt,
										aspect=0.5}] (-2,4.9) --  (0,4.9);
									\draw [decorate, 
									decoration = {calligraphic brace, mirror, 
										raise=0pt,
										aspect=0.75}] (-3,3.1) --  (-1,3.1);
									\draw [decorate, 
									decoration = {calligraphic brace, mirror, 
										raise=0pt,
										aspect=0.75}] (-4,2.1) --  (-2,2.1);
									\draw [decorate, 
									decoration = {calligraphic brace, mirror, 
										raise=0pt,
										aspect=0.5}] (-5,1.1) --  (-3,1.1);
									\endscope		
								\end{tikzpicture}
								\caption{Type C}
							\end{subfigure}
			\caption{The three types of partitions with Durfee triangle of size $4$.}
			\label{R_4Pic}
		\end{figure*}
		\begin{align}
			\textup{The generating function of partitions of type A}&= q^{10} \sum_{s=0}^{\infty} q^s \sum_{m=0}^{\infty} q^m \sum_{\ell=0}^{m+1} q^\ell \sum_{r=0}^{\ell+1} q^r \notag \\
			&= \frac{q^{10}(1+2q-q^3-2q^4+q^6)}{(1-q)^4(1+q)(1+q+q^2)}. \label{F4TA} \\
			\textup{The generating function of partitions of type B}&= q^{10} \sum_{s=0}^{\infty} q^s \sum_{r=1}^{s+1} q^r \sum_{m=0}^{\infty} q^m \sum_{\ell=1}^{m+1} q^\ell \notag \\
			&= q^{10} \left( \sum_{s=0}^{\infty} q^s \sum_{r=1}^{s+1} q^r \right)^2 \notag \\
			&= \frac{q^{12}}{(1-q)^4(1+q)^2}. \label{F4TB} \\
			\textup{The generating function of partitions of type C}&= q^{10}  \sum_{m=0}^{\infty} q^m \sum_{\ell=1}^{m+1} q^\ell \sum_{r=1}^{\ell+1} q^r \sum_{s=1}^{r+1} q^s \notag \\
			&= \frac{q^{13}(1+2q+q^2-2q^4-2q^5-q^6+q^7+q^8)}{(1-q)^4(1+q)^2(1+q^2)(1+q+q^2)}. \label{F4TC}
		\end{align}
		
		Simplifications similar to the ones done in Theorem \ref{F3T} are done to get \eqref{F4TA}, \eqref{F4TB} and \eqref{F4TC}. The generating functions of type A$^c$ and C$^c$ are the same as that of type A and C, so we multiply \eqref{F4TA} and \eqref{F4TC} by 2. Since B$^c$ is the same type as B, we ignore B$^c$. The partitions which are of type both A and A$^c$ are exactly when $\ell=0$ in \eqref{F4TA}, hence we subtract them. The only partitions which are of type C and C$^c$ are $(4,4,3,2)$, $(4,4,4,2)$, $(4,4,3,3)$, $(4,4,4,3)$ and $(4,4,4,4)$. Hence, we subtract off $q^{13}$, $2q^{14}$, $q^{15}$ and $q^{16}$ to get the required generating function. Thus,
		\begin{align*}
			\mathscr{F}_4(q) &= 2 \left(  \frac{q^{10}(1+2q-q^3-2q^4+q^6)}{(1-q)^4(1+q)(1+q+q^2)} \right) - \frac{q^{10}(1+q)}{(1-q)^2} + \left( \frac{q^{12}}{(1-q)^4(1+q)^2} \right) \\
			& \quad +2 \left(  \frac{q^{13}(1+2q+q^2-2q^4-2q^5-q^6+q^7+q^8)}{(1-q)^4(1+q)^2(1+q^2)(1+q+q^2)} \right) - q^{13}-2q^{14} -q^{15} -q^{16}.
		\end{align*}
		Simplify to get \eqref{ScrF4GF}.
	\end{proof}
	Now, we give a combinatorial proof of the generating function of $R_5(n)$, with some details curbed to reduce repetition.
	\begin{theorem}\label{F5T}
		For $|q|<1$, we have,
		\begin{align}
			\mathscr{F}_5(q) &= \sum_{n=1}^{\infty} R_5(n) q^n \notag \\
			&= \frac{q^{15}}{(1-q)^2(1-q^3)(1-q^4)(1-q^5)}\bigg( 1+4q+6q^2+8q^3+8q^4+4q^5+4q^6-5q^7 -5q^8 -10q^9 \notag\\ 
			 &\quad   -7q^{10}-5q^{11}+4q^{13}+6q^{14}+7q^{15}+2q^{17} -4q^{18}-q^{19}-q^{20}-q^{21}+q^{22}-q^{23}+q^{24} \bigg). \label{ScrF5GF}
		\end{align}		
	\end{theorem}
	\begin{proof}
		The partitions of a number with the Durfee triangle of size $4$ are at least of one of the type A, B, C, their conjugates A$^c$, B$^c$ or C$^c$, shown in Figure \ref{R_5Pic}. For type A, we take $m\geq0$, $\ell \leq m+1$, $r\leq \ell+1$, $s\leq r+1$ and $p\leq s+1$. For type B, we take $p\geq1$, $m\geq1$, $\ell\leq m+1$, $r\leq \ell+1$ and $s\leq r+1$, since  $p=0$ or $m=0$ gives a partition of type A or A$^c$ respectively. For type C, we take $m \geq 1$, $1\leq \ell \leq m+1$, $r\leq \ell+1$, $p\geq1$ and $1\leq s\leq p+1$, since $\ell=0$, $p=0$, $m=0$, and $s=0$ gives a partition of type B$^c$, A, A$^c$, and B respectively. We now find the generating functions of the types separately.  
		\begin{figure*}[h!]
			\centering
			\begin{subfigure}[t]{0.45\textwidth}
				\centering						
				\begin{tikzpicture}[description/.style={fill=white,inner sep=2pt}]
					\useasboundingbox (-5,-0.5) rectangle (1,4.5);
					\scope[transform canvas={scale=0.75}]
					\filldraw[black] (-6,1.5) circle (3pt);
					\filldraw[black] (-6,0.5) circle (3pt);
					\filldraw[black] (-6,2.5) circle (3pt);
					\filldraw[black] (-6,3.5) circle (3pt);
					\filldraw[black] (-6,4.5) circle (3pt);
					\filldraw[black] (-5,4.5) circle (3pt);
					\filldraw[black] (-5,2.5) circle (3pt);
					\filldraw[black] (-2.2,2.5) circle (1pt);
					\filldraw[black] (-2,2.5) circle (1pt);
					\filldraw[black] (-1.8,2.5) circle (1pt);
					\filldraw[black] (-3,2.5) circle (3pt);
					\filldraw[black] (-1,2.5) circle (3pt);
					\filldraw[black] (-4,2.5) circle (3pt);
					\filldraw[black] (-3.2,1.5) circle (1pt);
					\filldraw[black] (-3,1.5) circle (1pt);
					\filldraw[black] (-2.8,1.5) circle (1pt);
					\filldraw[black] (-5,1.5) circle (3pt);
					\filldraw[black] (-4,1.5) circle (3pt);
					\filldraw[black] (-2,1.5) circle (3pt);
					\filldraw[black] (-1.2,3.5) circle (1pt);
					\filldraw[black] (-1,3.5) circle (1pt);
					\filldraw[black] (-0.8,3.5) circle (1pt);
					\filldraw[black] (0,3.5) circle (3pt);
					\filldraw[black] (-2,3.5) circle (3pt);
					\filldraw[black] (-3,3.5) circle (3pt);
					\filldraw[black] (-5,3.5) circle (3pt);
					\filldraw[black] (-5,4.5) circle (3pt);
					\filldraw[black] (-4,3.5) circle (3pt);
					\filldraw[black] (-4,4.5) circle (3pt);
					\filldraw[black] (-3,4.5) circle (3pt);
					\filldraw[black] (-3,4.5) circle (3pt);
					\filldraw[black] (-0.2,4.5) circle (1pt);
					\filldraw[black] (0,4.5) circle (1pt);
					\filldraw[black] (0.2,4.5) circle (1pt);
					\filldraw[black] (1,4.5) circle (3pt);
					\filldraw[black] (-1,4.5) circle (3pt);
					\filldraw[black] (-2,4.5) circle (3pt);
					\filldraw[black] (-4.2,0.5) circle (1pt);
					\filldraw[black] (-4,0.5) circle (1pt);
					\filldraw[black] (-3.8,0.5) circle (1pt);
					\filldraw[black] (-5,0.5) circle (3pt);
					\filldraw[black] (-3,0.5) circle (3pt);
					\draw (-6,4.5) -- (-2,4.5);
					\draw (-6,4.5) -- (-6,0.5);
					\draw (-2,4.5) -- (-6,0.5);
					\node at (0,5.3) {m};
					\node at (-0.5,2.7) {$\ell$};
					\node at (-1.5,1.7) {r};
					\node at (-2.5,0.7) {s};
					\node at (-4,-0.3) {p};
 					\draw [decorate, 
					decoration = {calligraphic brace,
						raise=0pt,
						aspect=0.5}] (-1,4.9) --  (1,4.9);
					\draw [decorate, 
					decoration = {calligraphic brace, mirror, 
						raise=0pt,
						aspect=0.75}] (-2,3.1) --  (0,3.1);
					\draw [decorate, 
					decoration = {calligraphic brace, mirror, 
						raise=0pt,
						aspect=0.75}] (-3,2.1) --  (-1,2.1);
					\draw [decorate, 
					decoration = {calligraphic brace, mirror, 
						raise=0pt,
						aspect=0.75}] (-4,1.1) --  (-2,1.1);
					\draw [decorate, 
					decoration = {calligraphic brace, mirror, 
						raise=0pt,
						aspect=0.5}] (-5,0.1) --  (-3,0.1);
					\endscope		
				\end{tikzpicture}
				\caption{Type A}
			\end{subfigure}
			\begin{subfigure}[t]{0.45\textwidth}
				\centering
				\begin{tikzpicture}[description/.style={fill=white,inner sep=2pt}]
					\useasboundingbox (-5.5,-2) rectangle (1,4.5);
					\scope[transform canvas={scale=0.75}]
					\filldraw[black] (-6,1.5) circle (3pt);
					\filldraw[black] (-6,0.5) circle (3pt);
					\filldraw[black] (-6,2.5) circle (3pt);
					\filldraw[black] (-6,3.5) circle (3pt);
					\filldraw[black] (-6,4.5) circle (3pt);
					\filldraw[black] (-5,4.5) circle (3pt);
					\filldraw[black] (-5,2.5) circle (3pt);
					\filldraw[black] (-2.2,2.5) circle (1pt);
					\filldraw[black] (-2,2.5) circle (1pt);
					\filldraw[black] (-1.8,2.5) circle (1pt);
					\filldraw[black] (-3,2.5) circle (3pt);
					\filldraw[black] (-1,2.5) circle (3pt);
					\filldraw[black] (-4,2.5) circle (3pt);
					\filldraw[black] (-3.2,1.5) circle (1pt);
					\filldraw[black] (-3,1.5) circle (1pt);
					\filldraw[black] (-2.8,1.5) circle (1pt);
					\filldraw[black] (-5,1.5) circle (3pt);
					\filldraw[black] (-4,1.5) circle (3pt);
					\filldraw[black] (-2,1.5) circle (3pt);
					\filldraw[black] (-1.2,3.5) circle (1pt);
					\filldraw[black] (-1,3.5) circle (1pt);
					\filldraw[black] (-0.8,3.5) circle (1pt);
					\filldraw[black] (0,3.5) circle (3pt);
					\filldraw[black] (-2,3.5) circle (3pt);
					\filldraw[black] (-3,3.5) circle (3pt);
					\filldraw[black] (-5,3.5) circle (3pt);
					\filldraw[black] (-5,4.5) circle (3pt);
					\filldraw[black] (-4,3.5) circle (3pt);
					\filldraw[black] (-4,4.5) circle (3pt);
					\filldraw[black] (-3,4.5) circle (3pt);
					\filldraw[black] (-3,4.5) circle (3pt);
					\filldraw[black] (-0.2,4.5) circle (1pt);
					\filldraw[black] (0,4.5) circle (1pt);
					\filldraw[black] (0.2,4.5) circle (1pt);
					\filldraw[black] (1,4.5) circle (3pt);
					\filldraw[black] (-1,4.5) circle (3pt);
					\filldraw[black] (-2,4.5) circle (3pt);
					\filldraw[black] (-6,-1.3) circle (1pt);
					\filldraw[black] (-6,-1.5) circle (1pt);
					\filldraw[black] (-6,-1.7) circle (1pt);
					\filldraw[black] (-6,-0.5) circle (3pt);
					\filldraw[black] (-6,-2.5) circle (3pt);
					\draw (-6,4.5) -- (-2,4.5);
					\draw (-6,4.5) -- (-6,0.5);
					\draw (-2,4.5) -- (-6,0.5);
					\node at (0,5.3) {m};
					\node at (-0.5,2.7) {$\ell$};
					\node at (-1.5,1.7) {r};
					\node at (-3,0.7) {s};
					\node at (-6.6,-1.5) {p};
					\draw [decorate, 
					decoration = {calligraphic brace,
						raise=0pt,
						aspect=0.5}] (-1,4.9) --  (1,4.9);
					\draw [decorate, 
					decoration = {calligraphic brace, mirror, 
						raise=0pt,
						aspect=0.75}] (-2,3.1) --  (0,3.1);
					\draw [decorate, 
					decoration = {calligraphic brace, mirror, 
						raise=0pt,
						aspect=0.75}] (-3,2.1) --  (-1,2.1);
					\draw [decorate, 
					decoration = {calligraphic brace, mirror, 
						raise=0pt,
						aspect=0.5}] (-4,1.1) --  (-2,1.1);
					\draw [decorate, 
					decoration = {calligraphic brace, mirror, 
						raise=0pt,
						aspect=0.5}] (-6.3,-0.5) --  (-6.3,-2.5);
					\endscope
				\end{tikzpicture}
				\caption{Type B}
			\end{subfigure}
			\\
			\begin{subfigure}[t]{0.45\textwidth}
				\centering
				\begin{tikzpicture}[description/.style={fill=white,inner sep=2pt}]
					\useasboundingbox (-5.5,-2) rectangle (1.5,4.5);
					\scope[transform canvas={scale=0.75}]
					\filldraw[black] (-6,1.5) circle (3pt);
					\filldraw[black] (-6,0.5) circle (3pt);
					\filldraw[black] (-6,2.5) circle (3pt);
					\filldraw[black] (-6,3.5) circle (3pt);
					\filldraw[black] (-6,4.5) circle (3pt);
					\filldraw[black] (-5,4.5) circle (3pt);
					\filldraw[black] (-5,2.5) circle (3pt);
					\filldraw[black] (-2.2,2.5) circle (1pt);
					\filldraw[black] (-2,2.5) circle (1pt);
					\filldraw[black] (-1.8,2.5) circle (1pt);
					\filldraw[black] (-3,2.5) circle (3pt);
					\filldraw[black] (-1,2.5) circle (3pt);
					\filldraw[black] (-4,2.5) circle (3pt);
					\filldraw[black] (-5,-0.3) circle (1pt);
					\filldraw[black] (-5,-0.5) circle (1pt);
					\filldraw[black] (-5,-0.7) circle (1pt);
					\filldraw[black] (-5,1.5) circle (3pt);
					\filldraw[black] (-5,0.5) circle (3pt);
					\filldraw[black] (-5,-1.5) circle (3pt);
					\filldraw[black] (-1.2,3.5) circle (1pt);
					\filldraw[black] (-1,3.5) circle (1pt);
					\filldraw[black] (-0.8,3.5) circle (1pt);
					\filldraw[black] (0,3.5) circle (3pt);
					\filldraw[black] (-2,3.5) circle (3pt);
					\filldraw[black] (-3,3.5) circle (3pt);
					\filldraw[black] (-5,3.5) circle (3pt);
					\filldraw[black] (-5,4.5) circle (3pt);
					\filldraw[black] (-4,3.5) circle (3pt);
					\filldraw[black] (-4,4.5) circle (3pt);
					\filldraw[black] (-3,4.5) circle (3pt);
					\filldraw[black] (-3,4.5) circle (3pt);
					\filldraw[black] (-0.2,4.5) circle (1pt);
					\filldraw[black] (0,4.5) circle (1pt);
					\filldraw[black] (0.2,4.5) circle (1pt);
					\filldraw[black] (1,4.5) circle (3pt);
					\filldraw[black] (-1,4.5) circle (3pt);
					\filldraw[black] (-2,4.5) circle (3pt);
					\filldraw[black] (-6,-1.3) circle (1pt);
					\filldraw[black] (-6,-1.5) circle (1pt);
					\filldraw[black] (-6,-1.7) circle (1pt);
					\filldraw[black] (-6,-0.5) circle (3pt);
					\filldraw[black] (-6,-2.5) circle (3pt);
					\draw (-6,4.5) -- (-2,4.5);
					\draw (-6,4.5) -- (-6,0.5);
					\draw (-2,4.5) -- (-6,0.5);
					\node at (0,5.3) {m};
					\node at (-0.5,2.7) {$\ell$};
					\node at (-2,1.7) {r};
					\node at (-4.4,-0.5) {s};
					\node at (-6.6,-1.5) {p};
					\draw [decorate, 
					decoration = {calligraphic brace,
						raise=0pt,
						aspect=0.5}] (-1,4.9) --  (1,4.9);
					\draw [decorate, 
					decoration = {calligraphic brace, mirror, 
						raise=0pt,
						aspect=0.75}] (-2,3.1) --  (0,3.1);
					\draw [decorate, 
					decoration = {calligraphic brace, mirror, 
						raise=0pt,
						aspect=0.5}] (-3,2.1) --  (-1,2.1);
					\draw [decorate, 
					decoration = {calligraphic brace,  
						raise=0pt,
						aspect=0.5}] (-4.7,0.5) --  (-4.7,-1.5);
					\draw [decorate, 
					decoration = {calligraphic brace, mirror, 
						raise=0pt,
						aspect=0.5}] (-6.3,-0.5) --  (-6.3,-2.5);
					\endscope
				\end{tikzpicture}
				\caption{Type C}
			\end{subfigure}
			\caption{The three types of partitions with Durfee triangle of size $5$.}
			\label{R_5Pic}
		\end{figure*}
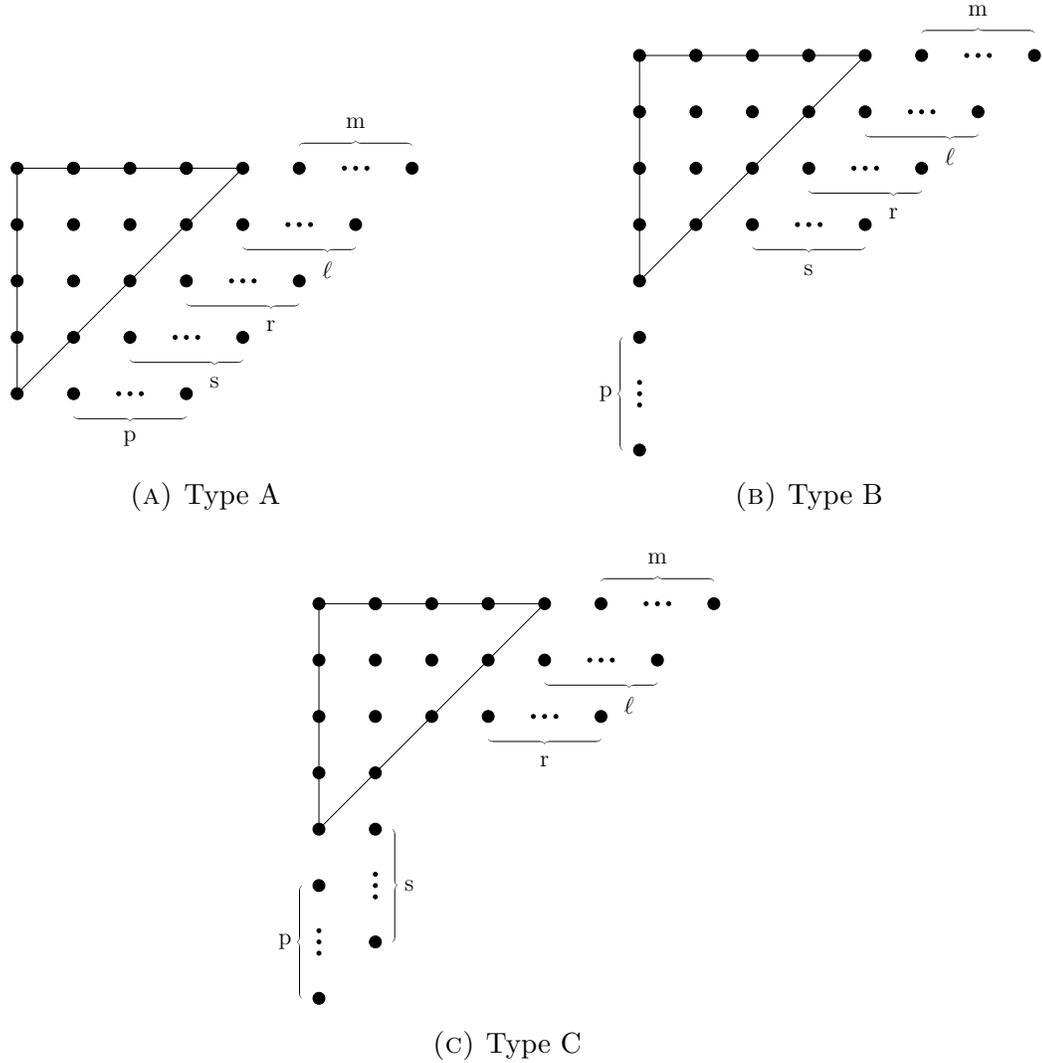
		\begin{align}
			&\textup{The generating function of partitions of type A}= q^{15} \sum_{m=0}^{\infty} q^m \sum_{\ell=0}^{m+1} q^{\ell} \sum_{r=0}^{\ell+1} q^r \sum_{s=0}^{r+1} q^s \sum_{p=0}^{s+1} q^p \notag \\
			&=\frac{q^{15}}{(1-q)^3(1+q)^2(1+q^2)(1-q^3)(1-q^5)} \bigg(1+4q+5q^2+3q^3-q^4-8q^5-9q^6-8q^7-2q^8+5q^9 \notag \\ 
			&\hspace{5.5 cm}+9q^{10}+7q^{11}+3q^{12}-q^{13}-4q^{14}-2q^{15}-2q^{16}+q^{20}\bigg). \label{F5TA}
		\end{align}
		\begin{align}
			&\textup{The generating function of partitions of type B}= q^{15} \sum_{m=1}^{\infty} q^m \sum_{\ell=0}^{m+1} q^{\ell} \sum_{r=0}^{\ell+1} q^r \sum_{s=0}^{r+1} q^s \sum_{p=1}^{\infty} q^p \notag \\
			&\hspace{1 cm} =\frac{q^{16}(q+3q^2+3q^3+q^4-3q^5-5q^6-4q^7-q^8+2q^9+3q^{10}+q^{11}+q^{12}-q^{16})}{(1-q)^4(1+q)^2(1+q^2)(1-q^3)}.  \label{F5TB}
		\end{align}
		\begin{align}
			&\textup{The generating function of partitions of type C}= q^{15} \sum_{m=1}^{\infty} q^m \sum_{\ell=1}^{m+1} q^{\ell} \sum_{r=0}^{\ell+1} q^r \sum_{p=1}^{\infty} q^p \sum_{s=1}^{p+1} q^s \notag \\
			&\hspace{4 cm} =\frac{q^{17}(q^2+3q^3+2q^4-2q^5-4q^6-2q^7+q^8+2q^9+q^{10}-q^{11})}{(1-q)^4(1+q)^2(1-q^3)}. \label{F5TC}
		\end{align}
		Simplifications similar to the ones done in Theorem \ref{F3T} are done to get \eqref{F5TA}, \eqref{F5TB} and \eqref{F5TC}. The partitions which are of type both A and A$^c$ are exactly when $m=0$ in type A. The partitions which are of type both B and B$^c$ are exactly when $\ell=0$ in type B. The partitions which are of type both C and C$^c$ are exactly when $r=0$ in type C. There are no possible repetitions apart from these three cases due to the definitions of the three types.
		\begin{align}
				&\textup{The generating function of partitions of type both A and A}^c \notag \\
				&= q^{15}\big( 1+4q+6q^2+7q^3+7q^4+5q^5+5q^6+3q^7+2q^8+q^9+q^{10} \big). \label{F5TAa}\\
				&\textup{The generating function of partitions of type both B and B}^c=   \frac{q^{17}(1+q-2q^2-q^3+q^6)}{(1-q)^4(1+q)}. \label{F5TBb}\\
				&\textup{The generating function of partitions of type both C and C}^c= \frac{q^{19}(1+q-q^2)^2}{(1-q)^4(1+q)^2}. \label{F5TCc}
		\end{align}
		The generating functions of type A$^c$, B$^c$ and C$^c$ are the same as that of type A, B and C respectively, so we multiply \eqref{F5TA}, \eqref{F5TB} and \eqref{F5TC} by 2. The generating functions of the partitions which are repeated are \eqref{F5TAa}, \eqref{F5TBb} and \eqref{F5TCc}, hence we subtract them off to get the required generating function \eqref{ScrF5GF}.
	\end{proof}
	\section{Exact formulas and Periodicity}\label{SecExa}
	In this Section, we obtain the exact formulas for $R_3(n)$ and $R_4(n)$ and further get their growth as $n \to \infty$. Due to the absence of infinite products in the generating functions, their exact formulas can be obtained, which directly can be used to obtain their growth rates without using the classical Circle method. Both of them have a polynomial growth rate unlike the Partition function which grows sub-exponentially. We prove the polynomial growth rates of $R_3(n)$ and $R_4(n)$ in Corollaries \ref{CorR3G} and \ref{CorR4G} respectively. The exact formula also gives the periodicity modulo $p$ for $p \in \mathbb{N}$ and $p\geq2$. 
	
	Let us define the basic notation before proving the Theorems. Let $f(x)$ and $g(x)$ be two real-valued functions. We say $f(x)\sim g(x)$ as $x\to \infty$ if,
	\begin{align*}
		\lim_{x\to \infty} \left(\frac{f(x)}{g(x)}\right)=1 .
	\end{align*}
	
	Let $f_1(q)$ and $f_2(q)$ be the generating functions of $a_1(n)$ and $a_2(n)$ respectively. Then, we say, $f_1$ and $f_2$ are congruent modulo $p$, i.e. $f_1(q) \equiv f_2(q) \textup{ (mod p)}$ if and only if $a_1(n) \equiv a_2(n) \textup{ (mod p)}$ for all $n \in \mathbb{N}$. We will use this notation in this paper hereon. We also use the exact formulas to obtain the periodic congruences modulo $p$ for any $p \in \mathbb{N}$ and $p\geq2$.
	\begin{theorem}\label{R3thmexk} 
		For $n>9$, such that $n=3m+a$ where $0\leq a\leq 2$, then, we have
		\begin{align}\label{R3formula}
			R_3(n) &=
			\begin{cases}
				6m^2 -15m+7, & \textup{if $a=0$,}\\
				6m^2 -11m+2, & \textup{if $a=1$,}\\
				6m^2 -7m-1, & \textup{if $a=2$.}
			\end{cases}
		\end{align}
		Further, for $n>9$ and for any fixed $p \in \mathbb{N}$ and $p\geq2$, $R_3(n) \textup{ mod }(p)$ is periodic in $n$ with period $3p$, i.e.,
		\begin{align*}
			R_3(n+3p) \equiv R_3(n) \textup{ mod } (p).
		\end{align*}
	\end{theorem}
	\begin{proof}
		From \eqref{ScrF3GF}, we have
		\begin{align}
			&\sum_{n=1}^{\infty} R_3(n) q^n \notag \\
			&= \frac{q^6(1+2q+q^2+2q^3-q^4-q^6-q^7+q^8)}{(1-q)^3(1+q+q^2)} \notag \\
			&= -7 -2q +q^2 +2q^3 +3q^4 +2q^5 -q^8 -q^9 +\frac{4}{3(1-q)^3}- \frac{7}{(1-q)^2} + \frac{110}{9 (1-q)} +  \frac{2(2+q)(1-q)}{9(1-q^3)}. \label{3inter}
		\end{align}		
		From the binomial theorem, for any $r \in \mathbb{N}$ and $|x|<1$, we have,
		\begin{align}
			(1-x)^{-r} = \sum_{k=0}^{\infty} \frac{(r)_k}{k!}x^k, \label{binom}
		\end{align}
		where $(r)_k$ is the Pochhammer symbol defined by $(r)_k:=(r)(r+1)\cdots (r+k-1)= \frac{\Gamma(r+k)}{\Gamma(r)}$.
		Using \eqref{binom} repeatedly with appropriate arguments, we get,
		\begin{align}
			&\frac{4}{3(1-q)^3}- \frac{7}{(1-q)^2} + \frac{110}{9 (1-q)} +  \frac{2(2+q)(1-q)}{9(1-q^3)} \notag \\
			&=\frac{2}{3} \sum_{k=0}^{\infty} (k+2)(k+1) q^k -  7\sum_{k=0}^{\infty}(k+1)q^k + \frac{110}{9} \sum_{k=0}^{\infty} q^k- \frac{2}{9} (q^2+q-2)  \sum_{k=0}^{\infty} q^{3k} \notag \\
			&= \sum_{k=0}^{\infty} \left( \frac{2}{3} (k+2)(k+1) -7(k+1) + \frac{110}{9} \right) q^k - \frac{2}{9} (q^2+q-2)  \sum_{k=0}^{\infty} q^{3k} \notag \\
			&= \sum_{k=0}^{\infty} \left( \frac{2}{3} k^2 -5k + \frac{59}{9} \right) q^k - \frac{2}{9} (q^2+q-2)  \sum_{k=0}^{\infty} q^{3k}\notag \\
			&= \sum_{m=0}^{\infty} \left( \frac{2}{3} (3m)^2 -5(3m) + \frac{59}{9} \right) q^{3m} +\sum_{m=0}^{\infty} \left( \frac{2}{3} (3m+1)^2 -5(3m+1) + \frac{59}{9} \right) q^{3m+1} \notag\\
			&\quad + \sum_{m=0}^{\infty} \left( \frac{2}{3} (3m+2)^2 -5(3m+2) + \frac{59}{9} \right) q^{3m+2} - \frac{2}{9} (q^2+q-2)  \sum_{k=0}^{\infty} q^{3k} \notag\\
			&= \sum_{m=0}^{\infty} \left( 6m^2 -15m + \frac{59}{9} \right) q^{3m} +\sum_{m=0}^{\infty} \left( 6m^2 -11m + \frac{20}{9} \right) q^{3m+1} \notag\\
			&\quad + \sum_{m=0}^{\infty} \left( 6m^2 -7m - \frac{7}{9} \right) q^{3m+2} - \frac{2}{9} (q^2+q-2)  \sum_{k=0}^{\infty} q^{3k}\notag \\
			&= \sum_{m=0}^{\infty} \left( 6m^2 -15m + 7 \right) q^{3m} +\sum_{m=0}^{\infty} \left( 6m^2 -11m + 2 \right) q^{3m+1} + \sum_{m=0}^{\infty} \left( 6m^2 -7m - 1 \right) q^{3m+2}. \label{3fin}
		\end{align}
		From \eqref{3inter} and \eqref{3fin}, we get the required exact formula for $R_3(n)$ for $n>9$. For any polynomial with integer coefficients, say $g(x):=a_0 + a_1 x + \cdots a_m x^m$, we can see that
		\begin{align*}
			g(x+p) &= a_0 + a_1 (x+p) + \cdots a_m (x+p)^m\\
			&\equiv (a_0 + a_1 x + \cdots a_m x^m) \textup{ mod (p)}\\
			&\equiv g(x) \textup{ mod (p)},
		\end{align*}
		since $(x+p)^k \equiv x^k$ mod $(p)$, for any $p \in \mathbb{N}$ such that $p\geq2$ and $k \in \mathbb{N}$, from binomial theorem. From \eqref{3inter} \eqref{3fin}, for $m>4$, we have seen that $R_3(3m) = 6m^2 -15m+7$ is a polynomial and hence $R_3(3(m+p)) \equiv R_3(3m) \textup{ mod (p)}$ for any $p \in \mathbb{N}$, $ p\geq2$. This same argument can be used for the other congruency classes to show that $R_3(n) \textup{ mod (p)}$ is a periodic function in $n$ with a period $3p$.
	\end{proof}
	\begin{corollary}\label{CorR3G}
		As $n \to \infty$, we have,
		\begin{align*}
			R_3(n) \sim \frac{2}{3}n^2.
		\end{align*}
	\end{corollary}
	\begin{proof}
		From Theorem \ref{R3thmexk}, we have,
		\begin{align*}
			R_3(3m) = 6m^2 - 15m +7.
		\end{align*}
	Replace $m$ by $\frac{n}{3}$ in the above equation to get the required result for the case $n=3m$. Then, repeat the above step for the other congruency classes modulo $3$ in Theorem \ref{R3thmexk} accordingly. Hence, $R_3(n)$ grows like $\frac{2}{3}n^2$ asymptotically, since, all the congruence classes have the identical asymptotic growth. 
	\end{proof}
	Next, we prove the analogous exact formula for $R_4(n)$.
	\begin{theorem}\label{R4thmexk} For $n>16$, such that $n=12m+a$ where $0\leq a\leq 11$, then, we have
		\begin{align}\label{R4formula}
			R_4(n) &=
			\begin{cases}
				192m^3 -264m^2+87m-3, & \textup{if $a=0$,}\\
				192m^3 -216m^2+44m+5, & \textup{if $a=1$,}\\
				192m^3 -168m^2+15m+5, & \textup{if $a=2$,}\\
				192m^3 -120m^2-12m+7, & \textup{if $a=3$,}\\
				192m^3 -72m^2-25m+4, & \textup{if $a=4$,}\\
				192m^3 -24m^2-36m+3, & \textup{if $a=5$,}\\
				192m^3 +24m^2-33m-2, & \textup{if $a=6$,}\\
				192m^3 +72m^2-28m-3, & \textup{if $a=7$,}\\
				192m^3 +120m^2-9m-5, & \textup{if $a=8$,}\\
				192m^3 +168m^2+12m-5, & \textup{if $a=9$,}\\
				192m^3 +216m^2+47m-3, & \textup{if $a=10$,}\\
				192m^3 +264m^2+84m+3, & \textup{if $a=11$,}\\
			\end{cases}
		\end{align}
		Further, for $n>16$ and for any fixed $p \in \mathbb{N}$ and $p\geq2$, $R_4(n) \textup{ mod }(p)$ is periodic in $n$ with period $12p$, i.e.,
		\begin{align*}
			R_4(n+12p) \equiv R_4(n) \textup{ mod } (p).
		\end{align*}
	\end{theorem}
	\begin{proof}
		From \eqref{ScrF4GF}, we have,
		\begin{align}
			&\sum_{n=1}^{\infty} R_4(n) q^n \\ &=\frac{q^{10}(1+4q+6q^2+7q^3+6q^4+2q^5-5q^7-5q^8-5q^9+q^{11}+3q^{12}+2q^{13}-q^{16})}{(1-q)^4(1+q)^2(1+q^2)(1+q+q^2)} \notag \\
			&= 3 -5q -5q^2 -7q^3 -4q^4 -3q^5 +2q^6 +3q^7 +5q^8 +5q^9 +4q^{10} +2q^{11} -q^{13} -2q^{14} -q^{15} -q^{16} \notag \\
			&\quad  - \frac{193}{18(1-q)} + \frac{965}{72(1-q)^2}- \frac{5}{(1-q)^3}+ \frac{2}{3(1-q)^4} - \frac{3}{2(1+q)} + \frac{1}{8(1+q)^2}+ \frac{1}{4(1+q^2)}- \frac{2(1-q^2)}{9(1-q^3)}. \label{4inter}
		\end{align}
		Using \eqref{binom} repeatedly with appropriate arguments, we get,
		\begin{align*}
			&-\frac{193}{18(1-q)} + \frac{965}{72(1-q)^2}- \frac{5}{(1-q)^3}+ \frac{2}{3(1-q)^4} - \frac{3}{2(1+q)} + \frac{1}{8(1+q)^2}+ \frac{1}{4(1+q^2)}- \frac{2(1-q^2)}{9(1-q^3)} \\
			&= -\frac{193}{18} \sum_{k=0}^{\infty} q^k  +\frac{965}{72} \sum_{k=0}^{\infty} (k+1) q^k -\frac{5}{2} \sum_{k=0}^{\infty} (k+2)(k+1) q^k +\frac{1}{9} \sum_{k=0}^{\infty} (k+3)(k+2)(k+1)q^k \\
			&\quad -\frac{3}{2} \sum_{k=0}^{\infty} (-1)^k q^k +\frac{1}{8} \sum_{k=0}^{\infty} (-1)^k (k+1) q^k +\frac{1}{4} \sum_{k=0}^{\infty} (-1)^k q^{2k} -\frac{2}{9} (1-q^2) \sum_{k=0}^{\infty} q^{3k} \\
			&= \sum_{k=0}^{\infty} \left( -\frac{193}{18}  +\frac{965}{72} (k+1) -\frac{5}{2}  (k+2)(k+1) +\frac{1}{9} (k+3)(k+2)(k+1)  -\frac{3}{2} (-1)^k +\frac{1}{8} (-1)^k (k+1) \right) q^k \\
			&\quad  +\frac{1}{4} \sum_{k=0}^{\infty} (-1)^k q^{2k} -\frac{2}{9} (1-q^2) \sum_{k=0}^{\infty} q^{3k} \\
			&= \sum_{k=0}^{\infty} \frac{1}{72} \bigg(8k^3 -132 k^2- (-1)^k 99+ 9(57+(-1)^k)k   -119\bigg) q^k  +\frac{1}{4} \sum_{k=0}^{\infty} (-1)^k q^{2k} -\frac{2}{9} (1-q^2) \sum_{k=0}^{\infty} q^{3k}.
		\end{align*}
		Similarly, as done in Theorem \ref{R3thmexk}, split the sums on the right-hand side of the above equation modulo 12 and simplify to get the required result \eqref{R4formula}.	 
		Further, from \eqref{R4formula}, for $m\geq2$, we have seen that $R_4(12m) = 	192m^3 -264m^2+87m-3$ is a polynomial and hence $R_4(12(m+p)) \equiv R_4(12m) \textup{ mod (p)}$. This same argument can be used for the other cases to show that $R_4(n) \textup{ mod (p)}$ is a periodic function in $n$ with a period $12p$.
	\end{proof}
	\begin{corollary}\label{CorR4G}
		As $n \to \infty$, we have,
		\begin{align*}
			R_4(n) \sim \frac{1}{9}n^3.
		\end{align*}
	\end{corollary}
	\begin{proof}
		From Theorem \ref{R4thmexk}, we have,
		\begin{align*}
			R_4(12m) =192m^3 - 264m^2 + 87m -3.
		\end{align*}
		Replace $m$ by $\frac{n}{12}$ in the above equation to get the required result for the case $n=12m$. Then, repeat the above step for the other congruency classes modulo $12$ in Theorem \ref{R4thmexk} accordingly. Hence, $R_4(n)$ grows like $\frac{1}{9}n^3$ asymptotically, since, all the congruence classes have the identical asymptotic growth. 
	\end{proof}
	\subsection{Parity and Congruences}\label{SecPar}
	In this Section, we study the parity of $R_k(n)$ for $k=3$, $4$ and $5$ using their generating functions. The same results can be obtained using their exact formulas too, which could be hard to handle due to their nature. We obtain the nature of parity of $R_5(n)$ without discovering its exact formula.
	\begin{theorem}\label{R3CTh}
		For $n>9$, we have,
		\begin{align*}
			R_3(n) \equiv 
			\begin{cases}
				0 \ (\textup{mod } 2), & \textup{if $n$ is odd,}\\
				1 \ (\textup{mod } 2), & \textup{if $n$ is even.}
			\end{cases}
		\end{align*}
	\end{theorem}
	\begin{proof}
		From \eqref{ScrF3GF}, we know,
		\begin{align*}
			\sum_{n=1}^{\infty} R_3(n) q^n &= \frac{q^6(1+2q+q^2+2q^3-q^4-q^6-q^7+q^8)}{(1-q)^3(1+q+q^2)} \\
			&\equiv  \frac{q^6(1+q^2+q^4+q^6+q^7+q^8)}{(1-q)^2(1-q^3)} \quad \textup{(mod 2)} \\
			&\equiv  \frac{q^6(1-q^6)}{(1-q)^2(1-q^3)} + \frac{q^6(q^2-q^8)}{(1-q)^2(1-q^3)}+ \frac{q^6(q^4-q^7)}{(1-q)^2(1-q^3)}  \quad \textup{(mod 2)} \\
			&\equiv  \frac{q^6(1+q^3)}{(1-q)^2} + \frac{q^8(1+q^3)}{(1-q)^2}+ \frac{q^{10}}{(1-q)^2}  \quad \textup{(mod 2)} \\
			&\equiv \frac{q^6+q^9+q^8+q^{10}+q^{11}}{1-q^2} \quad \textup{(mod 2)},
		\end{align*}
		since $\frac{1}{(1-q)^2} \equiv \frac{1}{(1-q^2)}$ (mod 2). For an even $n$, $q^6, q^8$ and $q^{10}$  contribute one each to the parity of $R_3(n)$ to make it odd. For an odd $n$, $q^9$ and $q^{11}$ contribute one each to the parity of $R_3(n)$ to make it even.
	\end{proof}
	\begin{theorem}\label{R4CTh}
		For $n>16$, we have,
		\begin{align*}
			R_4(n)\equiv
			\begin{cases}
				0 \ (\textup{mod } 2), & \textup{if $n \equiv 4,6$ (mod 8)}, \\
				1 \ (\textup{mod } 2), & \textup{if $n \equiv 0,1,2,3,5,7$ (mod 8)}.
			\end{cases}
		\end{align*}
	\end{theorem}
	\begin{proof}
		From \eqref{ScrF4GF}, we have,
		\begin{align*}
			&\sum_{n=1}^{\infty} R_4(n) q^n \notag \\
			&= \frac{q^{10}(1+4q+6q^2+7q^3+6q^4+2q^5-5q^7-5q^8-5q^9+q^{11}+3q^{12}+2q^{13}-q^{16})}{(1-q)^4(1+q)^2(1+q^2)(1+q+q^2)} \\
			&\equiv \frac{q^{10}+q^{13}+q^{17}+q^{18}+q^{19}+q^{21}+q^{22}+q^{26}}{(1-q)^4(1+q)^2(1+q^2)(1+q+q^2)}   \quad \textup{(mod 2)} \\
			&\equiv \frac{q^{10}+q^{13}+q^{17}+q^{18}+q^{19}+q^{21}+q^{22}+q^{26}}{(1-q)^3(1+q)^2(1+q^2)(1-q^3)}   \quad \textup{(mod 2)} \\
			&\equiv \frac{q^{10}(1-q^3)+q^{17}(1-q^3)+q^{18}(1-q^3)+q^{19}(1-q^3)+q^{20}(1-q^3)+q^{23}(1-q^3)}{(1-q)^3(1+q)^2(1+q^2)(1-q^3)}   \quad \textup{(mod 2)} \\
			&\equiv \frac{q^{10}+q^{17}+q^{18}+q^{19}+q^{20}+q^{23}}{(1-q)^2(1+q)(1-q^2)(1+q^2)}  \quad \textup{(mod 2)} \\
			&\equiv \frac{q^{10}+q^{17}+q^{18}+q^{19}+q^{20}+q^{23}}{(1-q)^2(1+q)(1-q^4)}  \quad \textup{(mod 2)}.
		\end{align*}
		Use the fact $\frac{1}{1-q^m} \equiv \frac{1}{1+q^m}$ (mod 2) for any $m \in \mathbb{N}$ and then multiply both the numerator of the right-hand side by $1+q$ in the next step to get,
		\begin{align*}
			\sum_{n=1}^{\infty} R_4(n) q^n &\equiv \frac{q^{10}+q^{17}+q^{18}+q^{19}+q^{20}+q^{23}}{(1-q)^2(1+q)(1+q^4)}  \quad \textup{(mod 2)} \\
			&\equiv \frac{(q^{10}+q^{17}+q^{18}+q^{19}+q^{20}+q^{23})(1+q)}{(1-q)(1+q)(1+q^2)(1+q^4)}  \quad \textup{(mod 2)} \\
			&\equiv \frac{q^{10}+q^{11}+q^{17}+q^{18}+q^{18}+q^{19}+q^{19}+q^{20}+q^{20}+q^{21}+q^{23}+q^{24}}{1-q^8} \quad \textup{(mod 2)} \\
			&\equiv \frac{q^{10}+q^{11}+q^{17}+q^{21}+q^{23}+q^{24}}{1-q^8} \quad \textup{(mod 2)}.
		\end{align*}
		For $n \equiv 0,1,2,3,5,7$ (mod 8) and $n>16$, the terms $q^{24}$, $q^{17}$, $q^{10}$, $q^{11}$, $q^{21}$, and $q^{23}$ contribute one each to the parity of $R_4(n)$ respectively, to make it odd. For $n \equiv 4,6$ (mod 8) and $n>16$, $R_4(n)$ are even since they do not appear on the right-hand side of the last equation.
	\end{proof}
	\begin{remark}
		For $n>16$ and $odd$, $R_4(n)$ is always odd. This is evident from the observation in Theorem \eqref{R4CTh}, that $R_4(n)$ is odd for all the four cases $n \equiv 1,3,5,7$ \textup{(mod 8)}.
	\end{remark}
	\begin{theorem}\label{R5CTh}
		For $n>25$, we have,
		\begin{align*}
			R_5(n)\equiv
			\begin{cases}
				0 \ (\textup{mod } 2), & \textup{if $n \equiv 5,7$ (mod 8)}, \\
				1 \ (\textup{mod } 2), & \textup{if $n \equiv 0,1,2,3,4,6$ (mod 8)}.
			\end{cases}
		\end{align*}
	\end{theorem}
		\begin{proof}
			From \eqref{ScrF5GF}, taking modulo $2$ of both the sides, we get,
			\begin{align*}
				&\sum_{n=1}^{\infty} R_5(n)q^n \\
				&\equiv \frac{q^{15}(1+q^7+q^8+q^{10}+q^{11}+q^{15}+q^{19}+q^{20}+q^{21}+q^{22}+q^{23}+q^{24})}{(1-q)^2(1-q^3)(1-q^4)(1-q^5)} \quad \textup{(mod 2)} \\
				&\equiv \frac{q^{15}((1-q^{15})+q^{5}(1-q^{5})+q^{7}(1-q^{5})+q^{8}(1-q^{5})+q^{11}(1-q^{5})+q^{12}(1-q^{5}))}{(1-q)^2(1-q^3)(1-q^4)(1-q^5)} \\
				&\quad + \frac{q^{15}(q^{13}(1-q^{5})+q^{15}(1-q^{5})+q^{16}(1-q^{5})+q^{17}(1-q^{5})+q^{18}(1-q^{5})+q^{19}(1-q^{5}))}{(1-q)^2(1-q^3)(1-q^4)(1-q^5)}  \quad \textup{(mod 2)} \\
				&\equiv \frac{q^{15}(1+q^5+q^7+q^8+q^{11}+q^{12}+q^{13}+q^{15}+q^{16}+q^{17}+q^{18}+q^{19})}{(1-q)^2(1-q^3)(1-q^4)} \quad \textup{(mod 2)}.
			\end{align*}
			Use the facts $\frac{1}{(1-q)} \equiv \frac{1}{(1+q)} \textup{ (mod 2)}$ and $\frac{1}{(1-q^4)} \equiv \frac{1}{(1+q^4)} \textup{ (mod 2)}$ and multiply the numerator and denominator by $(1+q^2)$ to get,
			\begin{align*}
				&\sum_{n=1}^{\infty} R_5(n)q^n \\
				&\equiv \frac{q^{15}(1+q^5+q^7+q^8+q^{11}+q^{12}+q^{13}+q^{15}+q^{16}+q^{17}+q^{18}+q^{19})(1+q^2)}{(1-q)(1+q)(1-q^3)(1+q^4)(1+q^2)} \quad \textup{(mod 2)} \\
				&\equiv \frac{q^{15}(1+q^2+q^5+q^8+q^9+q^{10}+q^{11}+q^{12}+q^{14}+q^{16}+q^{20}+q^{21})}{(1-q^8)(1-q^3)} \quad \textup{(mod 2)} \\
				& \equiv \frac{q^{15}((1-q^9)+(q^2-q^8)+(q^5-q^{11})+(q^{10}-q^{16})+(q^{12}-q^{21})+(q^{14}-q^{20}))}{(1-q^8)(1-q^3)} \quad \textup{(mod 2)}.
			\end{align*} 
			We use the identities $(1-q^6)=(1-q^3)(1+q^3)$ and $(1-q^9)=(1-q^3)(1+q^3+q^6)$ accordingly and simplify to get,
			\begin{align*}
				\sum_{n=1}^{\infty} R_5(n)q^n &\equiv \frac{1}{1-q^8}\left( q^{15}+q^{17}+q^{18}+q^{21}+q^{23}+q^{25}+q^{27}+q^{28}+q^{29}+q^{30}+q^{32}+q^{33} \right).
			\end{align*}
			For $n \equiv 0,2,3,4,6$ (mod 8) and $n>25$, the terms $q^{32}$, $q^{18}$, $q^{27}$, $q^{28}$ and $q^{30}$ contribute one each to the parity of $R_4(n)$ respectively, to make it odd. For $n \equiv 1$ (mod 8) and $n>25$, the term $q^{17}+q^{25}+q^{33}$ contributes three to the parity of $R_5(n)$, to make it odd. For $n \equiv 5,7$ (mod 8) and $n>25$, $q^{21}+q^{29}$ and $q^{15}+q^{23}$ contribute two each to the parity of $R_5(n)$, to make it even.
		\end{proof}
		\begin{remark}
			For $n>25$ and $even$, $R_5(n)$ is always odd. This is evident from the observation in Theorem \eqref{R5CTh}, that $R_5(n)$ is odd for all the four cases $n \equiv 0,2,4,6$ \textup{(mod 8)}.
		\end{remark}
	\section{Concluding remarks}
	We end the paper with the following concluding remarks:
	\begin{enumerate}
		\item Unlike the generating function of partitions with a size $k$ Durfee square $\mathscr{G}_k(q)$, the generating function of partitions with a size $k$ Durfee triangle $\mathscr{F}_k(q)$ seems to be non-trivial as can be seen from the increasing complexities of $\mathscr{F}_3(q), \mathscr{F}_4(q)$ and $\mathscr{F}_5(q)$, as shown in Section \ref{SectionGF}. Armin Straub and the author address this question in an upcoming paper \cite{DurfeeTri}.
		\item It is very curious and unclear as to why the parity congruences for $R_k(n)$ seem to hold for $n>k^2$. This suggests a \textit{combinatorial connection} between the Durfee triangle and the Durfee square of a partition.
		\item It should be possible to obtain the recurrence relations for $R_k(n)$ in general like \eqref{R2Rec} and \eqref{R3Rec}.
		\item The parity of $P(n)$, which is largely unexplored combinatorially, could be looked at using the parity of $R_k(n)$.
		\item The parity bias of $R_2(n)$ is same as that of $R_3(n)$, while the parity bias of $R_4(n)$ is same as that of $R_5(n)$. It is not evident why this is the case.
		\item From the initial few cases, it seems like the following conjecture holds true.
		\begin{conjecture}
			For any fixed $k \in \mathbb{N}$, as $n \to \infty$,
				\begin{align*}
					R_k(n) \sim c_k n^{k-1},
				\end{align*}
				for some constant $c_k$.
		\end{conjecture} 
		\item The exact formula for $R_k(n)$ is the polynomial expression for the partition function $P(n)$ which has been of interest for a lot of reasons including the faster computability. One must investigate if there could be an alternative way to obtain them. 
	\end{enumerate}
	
	\section{Acknowledgment}
	The author thanks Atul Dixit for his continuous and honest support. The author thanks James Sellers very much for his insight to obtain the exact formulas in \ref{R3formula} and \ref{R4formula}. The author thanks Krishnaswami Alladi for informing the reference of James Haglund, including \cite{BHR}. The author thanks Aviral Srivastava and Rajat Gupta for general discussions. The author thanks Armin Straub for pointing out a calculation error in one of the results.

	The list of OEIS sequences discussed: \seqnum{A325168}, \seqnum{A325188}. 
\end{document}